\documentclass[12pt]{amsart}

\usepackage{verbatim,amscd,amssymb}

\usepackage[active]{srcltx}
\usepackage{graphicx}

\usepackage{amsmath}
\usepackage{amsfonts}

\usepackage{inputenc}
\graphicspath{ {./images/} }
\usepackage{fontenc}

\usepackage{amsmath,amsthm,amssymb}
\usepackage{enumitem}

\usepackage{float}

\usepackage[mathscr]{eucal}
\usepackage{verbatim}
\usepackage{amscd}

\newcommand{\conv}{\underset{i\to \infty}{\longrightarrow}}

\newcommand{\Hc}{\mathcal{H}}
\newcommand{\Cc}{\mathcal{C}}
\newcommand{\umc}{\mathcal{UMC}}
\newcommand{\bd}{\mathrm{Bd}}
\newcommand{\itr}{\mathrm{int}}

\newcommand{\ol}{\overline}

\newcommand{\uml}{\mathrm{UML}}
\newcommand{\qml}{\mathrm{QML}}
\newcommand{\plane}{\mathbb{C}}

\def\int{\mathbb{Z}}
\newcommand{\disk}{\mathbb{D}}

\newcommand{\complex}{\plane}

\newcommand{\ucirc}{\mathbb{S}^1}
\newcommand{\uc}{\mathbb{S}^1}

\newcommand{\0}{\emptyset}

\newtheorem{thm}{Theorem}[section]

\newtheorem{lem}[thm]{Lemma}

\newcommand{\Aa}{\mathcal{A}}
\newcommand{\Ii}{\mathcal{I}}

\newtheorem{cor}[thm]{Corollary}

\theoremstyle{definition}
\newtheorem{defn}[thm]{Definition}

\newcommand{\udisk}{\mathbb{D}}

\def\R{\mathbb{R}}

\newcommand{\oc}{\ol{c}}
\newcommand{\oy}{\ol{y}}

\newcommand{\C}{\mathbb{C}}

\newcommand{\cdisk}{\ol{\mathbb{D}}}

\newcommand{\la}{\lambda}

\newcommand{\sm}{\setminus}
\newcommand{\A}{\mathcal{A}}

\newcommand{\lam}{\mathcal{L}}

\newcommand{\ch}{\mathrm{CH}}

\newcommand{\si}{\sigma}

\renewcommand\le{\leqslant}
\renewcommand\ge{\geqslant}
\def\0{\varnothing}

\begin{document} \title {Unicritical Laminations}

\author[S.~Bhattacharya]{Sourav~Bhattacharya}

\author[A.~Blokh]{Alexander~Blokh}

\author[D.~Schleicher]{Dierk~Schleicher}

\address[Sourav~Bhattacharya and Alexander~Blokh]
{Department of Mathematics\\ University of Alabama at Birmingham\\
Birmingham, AL 35294-1170}

\address[Dierk~Schleicher]{Aix-Marseille Universit\'e, Institut de Math\'ematiques de Marseille, 163 Avenue de Luminy Case 901, 13009
Marseille, France}

\email[Alexander~Blokh]{ablokh@uab.edu}

\email[Sourav~Bhattacharya]{sourav@uab.edu}

\email[Dierk~Schleicher]{dierk.schleicher@univ-amu.fr}



\date{January 18, 2021}
 	
\begin{abstract}
Thurston introduced \emph{invariant (quadratic) laminations} in his
1984 preprint as a vehicle for understanding the connected Julia sets
and the parameter space of quadratic polynomials. Important ingredients
of his analysis of the angle doubling map $\sigma_2$ on the unit circle
$\uc$ were the Central Strip Lemma, non-existence of wandering
polygons, the transitivity of the first return map on vertices of
periodic polygons, and the non-crossing of minors of quadratic
invariant laminations. We use Thurston's methods to prove similar
results for \emph{unicritical} laminations of arbitrary degree $d$ and
to show that the set of so-called \emph{minors} of unicritical
laminations themselves form a \emph{Unicritical Minor Lamination}
$\uml_d$. In the end we verify the \emph{Fatou conjecture} for the
unicritical laminations and extend the \emph{Lavaurs algorithm} onto
$\uml_d$.
\end{abstract}

\maketitle

	\section*{Introduction} 	

\everypar{\looseness=-1}

One of the most satisfactory strengths of polynomial dynamics is its successful interplay with symbolic
dynamics: even though the Julia sets and the Mandelbrot set have very complicated topological structure,
there are very simple combinatorial models that describe the original sets, often up to homeomorphism.

In the ``early days of modern holomorphic dynamics'', in the mid-1980's, Bill Thurston developed his theory
of quadratic minors for connected Julia sets of quadratic polynomials, and for their parameter space, the
\emph{Mandelbrot set} \cite{thu85}: he introduced the concept of \emph{invariant quadratic laminations} as
subsets of the closed unit disk and described their properties; each of these models a Julia set in the sense
that a natural quotient, called \emph{pinched disk} by Douady \cite{dou93}, is homeomorphic to the (filled-in)
Julia set if and only if the latter is locally connected; see also \cite{sch09}. Thurston showed that invariant
quadratic lamination is characterized by a unique leaf, called its \emph{minor leaf} (a set of two or possibly
one angles), and each angle is part of a unique minor leaf. Turning to parameter space, Thurston showed that
the union of all minor leaves forms another lamination, the \emph{quadratic minor lamination} $\qml$, that
models the Mandelbrot set via its own pinched disk, and is homeomorphic to it exactly when the Mandelbrot
set is locally connected, which continues to form a major conjecture in the field. Douady, Hubbard, and
Lavaurs showed that $\qml$ and hence the Mandelbrot set can be constructed by a very simple algorithm,
known as the \emph{Lavaurs algorithm}, that inductively constructs all periodic minor leaves in the
order of their period, and so that the closure of all periodic leaves is all of $\qml$.

Thurston already started the question to extend his theory of quadratic minor laminations to
higher degrees. For arbitrary degrees $d\ge 3$, this is a substantial and difficult task. The
first question that Thurston viewed as fundamental, whether ``wandering triangles'' exist,
was resolved in the affirmative in \cite{bo08}; the non-existence of wandering triangles is
one of the key lemmas in his quadratic theory (see Theorem II.5.2 from \cite{thu85}), so this
creates a substantial difficulty. A second difficulty lies in the fact that parameter space of
degree $d$ polynomials has complex dimension $d-1$, so the one-dimensional theory of parameter
spaces fails fundamentally. Even the combinatorial description of cubic polynomials is most
difficult; see for instance \cite{mil08} and \cite{KafflsThesis}. Thurston's work on general
laminations of degree $d$ was recently published in \cite{thu19}

\looseness=-1 In this manuscript, we focus on a particular class of
degree $d$ laminations called \emph{unicritical}: unicritical
polynomials can always be pa\-ra\-me\-terized as $z^d+c$, so they have
a one-dimensional parameter space; their connectedness locus is often
called the \emph{Multibrot set}, and it is well known to have profound
similarities to the Mandelbrot set (see for instance \cite{sch04}).
Fundamental properties of unicritical laminations have been known for a
long time, in particular the non-existence of wandering triangles
\cite{lev98, sch00}.

We develop unicritical laminations for degree $d$ in analogy to Thurston's quadratic laminations.
While it has been expected frequently that many of his original methods can be made to work in this
setting, we are not aware of a systematic treatment that describes unicritical laminations in detail.
Our main results contain the following:

\noindent (1) we describe unicritical laminations; in particular, we show that each of its gaps is periodic, and
its boundary either consists of finitely many leaves (modeling finitely many rays landing at a periodic
point in the Julia set), or it is a Cantor set of leaves that returns back to itself either with degree
$1$ or degree $d$ (modeling a Siegel disk or a Cremer point, respectively an attracting periodic orbit);

\noindent (2) we show that each unicritical lamination has a unique minor leaf, which may or may not be degenerate, that
every angle occurs as the minor leaf of a unicritical lamination, and the union of all such minor leaves of
given degree  $d\ge 2$ forms another lamination, called the unicritical minor lamination (Theorem~\ref{t:umlq})
associated with an equivalence relation on the unit circle;

\noindent (3) in the unicritical minor lamination, the boundary of each gap is either finite (modeling finitely many preperiodic
parameter rays landing at a common point), or it is a Cantor set (modeling a hyperbolic component);

\noindent (4) each unicritical minor lamination is the closure of its periodic leaves (all of which are non-degenerate),
and these leaves can be constructed recursively by increasing periods by a simple algorithm analogous to
Lavaurs' algorithm (Theorem~\ref{Lavaurs:lemma} and Section \ref{Sec:Lavaurs}).

Section 1 contains
basic definitions and results. In Section
2 we study \emph{unicritical $\si_d$-invariant laminations} and show
that the minors of \emph{unicritical laminations}
form a \emph{lamination} called \emph{Unicritical Minor
Lamination of  degree $d$} and denote it by $\uml_d$. In Section 3 we
study the basic properties of $\uml_d$. In Section 4 we prove that $\uml_d$ is a
\emph{q-lamination}. In Section 5, we prove additional facts that
help us to devise an algorithm to construct $\uml_d$ similar to
the \emph{Lavaurs algorithm} \cite{lav89} for the quadratic
case (see Section 6).

	\section{Preliminaries} \label{Sec:Prelim}
	
	\subsection{Laminational equivalence relations}
	
	Let $ \hat{\complex} $ be the Riemann sphere. For a compactum
	$X\subset \complex $, let $U^{\infty}(X)$ be the component of
	$\hat{\complex} \backslash X $ containing $\infty$. For $X$ connected,
	let $\Psi_{X}: \hat{\complex} \backslash
	\overline{\disk} \rightarrow U^{\infty}(X)$ be a Riemann map with
	$\Psi_{X}(\infty)=\infty$, and $\Psi_{X}(z)$ tending to a positive real
	limit as $z\rightarrow \infty$.
	Let $P$ be a monic polynomial of degree $d\geq 2$.
Consider the Julia set $J_P$ and the filled
	Julia set $K_P$ of $P$. Set $\theta_d$ be the map $z^d|_{U^\infty(\cdisk)}$. If $J_P$ is connected,
	$\Psi \circ \theta_d = P \circ \Psi$ \cite{dh8485, mil00}. If
	$J_P$ is locally connected, $\Psi $ extends to a continuous
	map $\overline{\Psi}:\hat{\complex} \backslash  \disk \rightarrow
	\overline{\hat{\complex} \backslash K_P }$, and $\overline{\Psi} \circ
	\theta_d = P \circ \ol{\Psi} $ on $\C\sm \disk$.
	A continuous surjection $\overline{\Psi}|_{\uc}=\psi:
\uc
	\rightarrow J_P$ is called the \emph{Caratheodory loop}. We will write $\sigma_d$ for $\theta_d|_{\uc}$. Define an
	equivalence relation $\sim_P$ on $\uc$ by $x \sim_P y$ iff $\psi(x)=\psi(y)$ and call $\sim_P$ the \emph{laminational
		equivalence relation (generated to the polynomial $P$)}.
    The
	map $\psi$ semi-conjugates $\sigma_d$ and $P|_{J(P)}$, and
	$\sim_P$ is invariant; $\sim_P$-classes have
	pairwise disjoint convex hulls. The quotient space $\uc /\sim_P=J_{\sim P}$ is called the \emph{topological Julia set}.
	Clearly, $J_{\sim P}$ is homeomorphic to $J_P$. The map $f_{\sim P} :
	J_{\sim P}\rightarrow J_{\sim P}$, induced by $\sigma_d$, called the
	\emph{topological polynomial}, is
	topologically conjugate to $P|_{J_P}$.
	Define an equivalence relation $\sim $ on $\uc$ similar
	to $\sim_P$ but with no references to polynomials.
	
	\begin{defn}\label{Lam:equivalence}
		An equivalence relation $\sim$ on the unit circle $\uc$ is said
		to  be \emph{laminational} if :
		\begin{enumerate}
			\item the graph of $\sim$ is a closed subset of $ \uc
\times \uc $;
			\item convex hulls of disjoint equivalence classes are disjoint;
			\item each equivalence class of $\sim$ is finite.
		\end{enumerate}
The equivalence relation $\sim$ is ($\sigma_d$)-\emph{invariant} if:
		
		\begin{enumerate}
			
			\item $\sim$ is \emph{forward invariant}: for a
$\sim$-class
			$[g]$, the set $\sigma_d([g])$ is a $\sim$-class too;
			
			\item for any $\sim$-class $[g]$, the map $ \tau =
			\sigma_d|_{[g]}$ extends to $\uc$ as an  orientation
preserving
			covering map $\hat{\tau}$ such that $[g]$ is the full preimage of
			$\tau([g])$ under the covering map  $\hat{\tau}$.
		\end{enumerate}
	\end{defn}

Assume that
$\sim$ is a $\sigma_d$-invariant laminational equivalence relation.
Consider the \emph{topological Julia set}
	$\uc/\sim = J_{\sim}$,  and the \emph{topological polynomial}
	$f_{\sim}: J_{\sim} \rightarrow J_{\sim}$ induced by $\sigma_d$. We can use Moore's Theorem to embed
	$J_{\sim}$ into $\complex$ and to extend the quotient map $
	\psi_{\sim} : \uc \to J_{\sim}$ to a map $ \psi_{\sim} : \complex
	\to \complex $ with the only non-degenerate \emph{fibers} (point
preimages) being the convex
	hulls of non-degenerate $\sim$-classes. A \emph{Fatou domain} of
	$J_{\sim}$ ($f_{\sim}$) is a component of
	$\complex \backslash J_{\sim} $. If $U$ is a periodic \emph{Fatou
		domain} of $f_{\sim}$ of period $n$, $ f^n_{\sim}|_{\bd(U)}$ is
	conjugate to an irrational rotation of $\uc$ (then $U$ is
called a \emph{Siegel domain}) or to
    $\sigma_k$
	for some $k>1$ \cite{bl02}. The complement $K_{\sim}$ of $U(J_{\sim})$ is the
	\emph{filled topological Julia set}, i.e. the union of $J_{\sim}$ and its bounded
	\emph{Fatou domains}. If $\sim$
	is fixed, we omit it from the notation. We
	consider $f_{\sim}$ as a self-mapping of $J_{\sim}$.

	\begin{defn}\label{leaf}
	For a $\sim$-class $A$, call an edge $\overline{ab}$ of $\ch(A)$($\ch$ denotes \emph{convex hull}) a
		\emph{($\sim$-)leaf}. All points of $\uc$ are called
		\emph{degenerate} $\sim$-leaves.
	\end{defn}
	
	The limit of a converging sequence  of $\sim$-leaves
	$\sim$ is a $\sim$-leaf (the family of all
	$\sim$-leaves is closed); the union of all $\sim$-leaves is a continuum. For $X\subset \cdisk $ with the
	property $X =\ch(X\cap \uc)$, set
$\sigma_d(X)=\ch(\sigma_d(X
	\cap \uc))$; then for any leaf $\ell$ of $\sim$, the set
	$\sigma_d(\ell)$ is a leaf.
 	
	\begin{defn}\label{d:q-lamination}
		The set $\mathcal{L}_{\sim}$ of all leaves of $\sim$ is called the
		\emph{lamination generated by} $\sim$. A lamination generated
by some laminational equivalence relation is called a
\emph{q-lamination}.
	\end{defn}

	\subsection{Laminations as defined by Thurston}\cite{thu85}
By a \emph{prelamination} we mean a collection $\mathcal{L}$ of pairwise unlinked chords
	in $\cdisk$, called \emph{leaves}. For a collection $\mathcal{R}$ of sets, we denote the union of all sets from $\mathcal{R}$ by $\mathcal{R}^{+}$. If all points of the
	circle are elements of $\mathcal{L}$ (seen as \emph{degenerate leaves})	and $\lam^+$ is closed in $\complex$, then we	call $\mathcal{L}$ a \emph{lamination}. Hence, a lamination can be obtained
	by closing a \emph{prelamination} and adding all points of $\uc$ viewed as degenerate leaves. If $ \ell=\overline{ab}$ is a leaf,let
	$\sigma_d(\ell)$ be the chord with endpoints $\sigma_d(a)$ and
	$\sigma_d(b)$. If $\sigma_d(a)=\sigma_d(b)$, call $\ell$ a
	\emph{critical leaf} and $\sigma_d(a)$ a \emph{critical value}. If
	$\sigma_d^* : \mathcal{L}^+ \to \overline{\disk}$ is the linear
	extension of $\sigma_d$ over all the leaves in $\mathcal{L}$, then
$\sigma_d^*$ is continuous. Also, $\sigma_d$ is
	locally one-to-one on $\uc$, and $\sigma_d^*$ is one-to-one on any
	given non-critical leaf. If $\mathcal{L}$ is a lamination, then
	$\mathcal{L}^+$ is a continuum.

We call the closure of a component of $\disk - \mathcal{L}^{+}$ a \emph{gap}. For a gap $G$, call $G\cap \uc$ the \emph{basis} of $G$ and the points of $G\cap \uc$ \emph{vertices} of
$G$. A gap $G$ is \emph{infinite} if and only if $G\cap \uc$ is infinite.

	\begin{defn}\label{relation:lamination}
		Let $\mathcal{L}$ be a lamination. The equivalence relation
		$\approx_{\mathcal{L}}$  on $\uc$ defined by declaring that $x
		\approx_{\mathcal{L}}$ if and only if there exists a finite
concatenation of leaves
		of $\mathcal{L}$ joining $x$ and $y$ will be called the
\emph{equivalence
		relation generated by $\mathcal{L}$}.
	\end{defn}

Recall that q-laminations were defined above in Definition
\ref{d:q-lamination}. Equivalently, \emph{a lamination $\mathcal{L}$ is a
q-lamination if and only if the equivalence relation
$\approx_{\mathcal{L}}$ is an invariant laminational equivalence
relation as defined in Definition $\ref{Lam:equivalence}$ and
$\mathcal{L}$ consists exactly of the edges of the convex hulls of all
$\approx_{\mathcal{L}}$-classes together with all points of $\uc$.}

A lamination $\mathcal{L}$ is \emph{Thurston $\si_d$-invariant} if it is (1) \emph{forward $\si_d$-invariant}: for any leaf
$\ell=\overline{ab} \in \mathcal{L}$,  either $\sigma_d(a)=\sigma_d(b)$ or $\sigma_d(\ell)=\overline{\sigma_d(a) \sigma_d(b)}
\in \mathcal{L}$, (2) \emph{backward $\si_d$-invariant}: for any leaf
$\ell=\overline{ab} \in \mathcal{L}$ there is a collection of $d$ disjoint leaves in $\mathcal{L}$ each joining a preimage of $a$ to a
preimage of $b$,
(3) \emph{gap $\si_d$-invariant}:  for any gap $G$ in $\mathcal{L}$, the convex hull $H$ of $\sigma_d(G\cap \uc)$ is a gap of $\lam$, a leaf of $\lam$, or a point of $\uc$. If $H$ is a gap, the
map $\sigma_d^*|_{\bd(G)}: \bd(G) \to \bd(H)$ is the
composition of a monotone and a covering map to
$\bd(H)$ with positive orientation, and
the image of a point moving in the positive direction
around $\bd(G)$ moves (non-strictly) in the positive direction around $\bd(H)$.

A gap $G$ of a $\si_d$-invariant lamination is called $\sigma_d$ -\emph{critical}
if for each  $y \in \sigma_d(G)$ the set $\sigma_d^{-1}(y) \cap G$ consists of at least 2 points. A set of
 critical leaves $\mathcal C=\{\oc_i\}_{i=1}^{d-1}$ is \emph{full} if $\mathcal C^+$ contains no polygons.
Theorem \ref{prop:sat} shows how to generate \emph{Thurston invariant laminations:}

\begin{thm} \cite{thu85}\label{prop:sat}
	Let $\lam$ be a lamination without critical gaps which satisfies
    the following conditions:
		
		\begin{enumerate}
			
			\item $\lam$ is forward $\si_d$ invariant and

			\item there is a full collection of critical leaves in
$\mathcal{L}$.
			
		\end{enumerate}
		
		Then there exists a $\si_d$-invariant lamination $S(\lam)$
        which contains $\lam$ and is obtained
		by taking pullbacks of the leaves in $\lam$.
	\end{thm} 	

\begin{defn}\label{d:sat} The lamination $S(\lam)$ is called the \emph{saturation} of
$\lam$.
\end{defn}

	\subsection{Classification of gaps}
	\begin{defn}[Periodic and (pre)periodic gaps]\label{Periodic:gaps}
Let $G$ be a gap of a $d$-invariant lamination $\mathcal{L}$. $G$ is
		\emph{(pre)periodic} if
$\sigma_d^{m+k}(G)=\sigma_d^{m}(G)$ for
		some $m \geq 0 $, $k >0 $; if $m$,$k$ are chosen to be minimal, then
		$G$ is said to be \emph{preperiodic} if $m>0$ or periodic (of
\emph{period} $k$) if
		$m=0$. If the period of $G$ is 1, then $G$ is said to be
\emph{invariant.}
		Define \emph{pre-critical} and \emph{(pre)critical} objects
similarly to
		(pre)periodic objects defined above. 	
	\end{defn}
	
	Consider infinite periodic gaps of $\sigma_d$-invariant
laminations.
	By \cite{kiw02}, infinite gaps are eventually mapped onto
	periodic infinite gaps. We state (without a proof)
	folklore results about (pre)periodic (in particular, infinite)
	gaps (for q-laminations see \cite{bl02}, see also \cite{bopt20}). Any edge of a (pre)periodic gap is either (pre)periodic or precritical.

	
	\begin{defn}[Fatou gaps]\label{Fatou:gap}
		An infinite gap $G$ is said to be a \emph{Fatou} gap if its
basis $G\cap \uc$ is uncountable.
	\end{defn}

If $G$ is a Fatou gap, $G\cap \uc$ can be represented as
the union of a non-empty perfect set and a countable set. By
\cite{kiw02} $G$ is \emph{(pre)periodic} under  $\sigma_d$. If $G$ is a
\emph{Fatou gap}, then the map $\psi_{G}:\bd(G) \to \uc $, collapsing
all edges of $G$ to points, monotonically maps $\bd(G)$ onto $\uc$.

\begin{thm}[Siegel gaps and Fatou gaps of degree
		$k>1$]\label{Classify:Fatou:gaps} Let $G$ be a
periodic Fatou gap of minimal period $n$. Then $\psi_{G}$
semiconjugates $\sigma_d^n|_{\bd(G)}$ to a map
$\hat{\sigma_G}=\hat{\sigma}:\uc\to \uc $ so that either (1)
$\si_n|_{\bd(G)}$ is of degree $k\ge 2$ and $\hat{\sigma}=\sigma_k:\uc
\to \uc,$ or (2) $\si_n|_{\bd(G)}$ is of degree one and $\hat{\sigma}$
is an irrational rotation.
\end{thm}

If Theorem \ref{Classify:Fatou:gaps}(1) holds, $G$ is called a periodic
gap \emph{of degree $k$}, if Theorem \ref{Classify:Fatou:gaps}(2)
holds, $G$ is called a \emph{Siegel} gap. A (pre)periodic gap
eventually mapped to a periodic gap of degree $k$ (Siegel) is also said
to be \emph{of degree $k$ (Siegel)}.

\subsection{Sibling ($\si_d$-)invariant laminations} \label{sibling:invariant}

\emph{Sibling invariant laminations} were introduced in \cite{bmov13}. This notion is slightly more restrictive than that of \emph{Thurston invariant lamination}. However, \emph{sibling invariant geolaminations form a
closed set, include all q-laminations, and are Thurston invariant}. For our purposes it suffices to consider sibling invariant laminations only. A major advantage of working with them is that they are defined exclusively through properties of their leaves.\emph{A lamination $\mathcal{L}$ is sibling ($\si_d$-)invariant} if: \emph{(1) for each $\ell \in \mathcal{L}$, we have $\sigma_d(\ell) \in \mathcal{L}$}
 \emph{(2) for each $\ell \in \mathcal{L}$ so that $
\sigma_d(\ell)$ is a non-degenerate leaf,  there exist $d$ disjoint leaves $\ell_1$, $\ell_2$, $\dots$, $\ell_d$, in $\mathcal{L}$ so that $\ell=\ell_1$ and $\sigma_d(\ell_i)=\sigma_d(\ell)$ for all $
i=1,2,\dots,d$.} \emph{(3) for each $\ell \in \mathcal{L}$,  there exists $\ell_1 \in  \mathcal{L}$ so that $\sigma_d(\ell_1)=\ell$}.
	
	\begin{thm}[\cite{bmov13}]\label{t:space}
The space of all sibling $\sigma_d$-invariant lamination is compact.All q-laminations are sibling $\sigma_d$-invariant.
\end{thm}
	
	In what follows instead of \emph{sibling $\sigma_d$-invariant	laminations } we talk of  \emph {$\sigma_d$-invariant laminations}. Also, we talk interchangeably about leaves (gaps) of $\sim$ or of $ \mathcal{L}_{\sim}$. 	

\section{Unicritical invariant laminations: basic properties} \label{Sec:Lamin_Basic}

\begin{defn}\label{unicritical}
 A $\si_d$-invariant lamination is said to be \emph{unicritical}
if it has a unique critical set $\mathcal{C}(\lam)$ (which maps forward
$d$-to-$1$).
\end{defn}

From now on let $\lam$ be a
$\si_d$-invariant unicritical lamination. Let us discuss properties of $\lam$. Clearly, $\mathcal{C}(\lam)=\Cc$ is symmetric with respect to
rotations by $\frac{1}{d}$. The set $\Cc$ can be a regular $d$-gon
whose vertices cut the circle into $d$ arcs of length $\frac{1}{d}$
each. Otherwise $\Cc$ has more than $d$ sides and each circle arc
complementary to $\Cc\cap \uc$ is shorter than $\frac{1}{d}$.
Thus, $\Cc$ and in general any \emph{unicritical lamination}
$\mathcal{L}$ of degree $d$ has a \emph{$d$-fold rotational symmetry}.
In particular, every leaf $\ell\in \mathcal{L}$ has $d-1$ distinct
\emph{siblings}, all of equal length, placed symmetrically with respect
to the origin. A \emph{unicritical lamination} $\mathcal{L}$ has at
least $d$ leaves of maximal length called the \emph{majors} of
$\mathcal{L}$. Indeed, any unicritical lamination has the $d$-fold
rotational symmetry and so a major has $d-1$ distinct siblings. However
a priori there might exist another group of $d$ siblings of the same
length, etc. We will show below that because of the properties of the
length function this is not possible. In particular, this will imply
that there is only one group of $d$ sibling majors, and, therefore, a
unique image of a major of a lamination $\lam$ called the \emph{minor}
of $\mathcal{L}$.

Let $(a,b)$ be the arc of $\uc$ from the point $a$ to the point $b$
traversed anticlockwise; let $|(a,b)|$ be the length of this arc.
Define the \emph{length $|\overline{ab}|$ of a chord $\overline{ab}$}
as $|\overline{ab}|=\min \{|(a, b)|, |(b, a)|\}$.
The maximum length of a chord is $\frac{1}{2}$. Call a chord
$\ell=\overline{ab}$ \emph{critical} if $\sigma_d(a) = \sigma_d(b)$. The
length of a critical chord $\ell$ is $|\ell|=\frac{j}{d}$ where $j\in
\{1, \dots, t|t= [\frac{d}{2}]\}$ where $[\frac{d}{2}]$ denotes the
greatest integer not exceeding $\frac{d}{2}$. Evidently,
$|\sigma_d(\ell)|$ as a function of $|\ell|$ can be described as a
``sawtooth'' $d$-to-$1$ map $\psi:[0, \frac12]\to [0, \frac12]$
with $d$ segments of monotonicity (\emph{laps})
and slope $d$ or $-d$ on each of them; increasing and decreasing laps
alternate and the first lap $[0, \frac{1}{2d}]$ is increasing. In what
follows set $\psi(t)$ to be the \emph{length function} reflecting how
$|\si_d(\ell)|$ depends on $|\ell|$. The line $y=x$ intersects the
graph of $ |\sigma_d(\ell)|$ with respect to $|l|$ in the interval $
[0, \frac{1}{d}]$ at the points $0$ and $\frac{1}{d+1}$.
\begin{lem} \label{length:function}	
	A leaf $\ell$ of length $|\ell| < \frac{1}{d+1}$  will keep
	increasing in length under iteration of $\si_d$ until for some
	iterate we have $|\sigma_d^j(\ell)|\geq \frac{1}{d+1}$. Thus,
$|M(\lam)|\ge \frac{1}{d+1}$. Hence there are exactly $d$ majors of
$\lam$ that are all siblings and a unique minor $m(\lam)$ such that
$|m(\lam)|\le \frac{1}{d+1}$.
\end{lem}

\begin{proof}
We leave to the reader the proof of initial claims of Lemma
\ref{length:function} dealing with how the length of a leaf grows. This
implies that $|M(\lam)|\ge \frac{1}{d+1}$. If there exists another
group of $d$ majors which have the same image and are of the length
$|M(\lam)|$, this will imply that $1\ge 2d|M(\lam)|\ge
\frac{2d}{d+1}>1$, a contradiction (recall that $d\ge 2$).
\end{proof}

For a non-diameter chord $n=\ol{ab}$, the smaller of the two arcs into
which $n$ divides $\uc$, is denoted by $H(n)$. Also, whenever we talk
about circle arcs $(x, y)$ of any length we always mean that the
movement from $x$ to $y$ inside $(x, y)$ is in the positive direction.
In the case of $H(n)$ we always set $n=\ol{ab}$ (thus, the movement
inside $H(n)$ from $a$ to $b$ is in the positive direction). In that
case we refer to points of $H(n)$ as located \emph{``under''} $n$. In
this setting $a$ is called the \emph{initial} point of $n$ (and $H(n)$)
while $b$ is called the \emph{terminal} point of $n$ (and $H(n)$).
Thus, for a non-diameter chord $\ell$ we can write
$\ell=\ol{i(\ell)t(\ell)}$.

Call the \emph{region} of the disk bounded by $m$ and $H(m)$ the
\emph{region bounded by $m$} and denote it by $\mathcal{R}(m)$. Given
two non-diameter chords $m$ and $n$, write $m\succ n$ if $n\subset
\mathcal{R}(m)$, and say that $n$ is a \emph{successor of $m$}. If $m,
n\in \uml_d, m\succ n $ and there exists no $\ell \in \uml_d$
satisfying $m\succ \ell\succ n$, then $n$ is called an \emph{immediate
successor of $m$}.

Given a minor $m=\ol{ab}$, call the arc $H(m)$ a \emph{minor} arc. Let
$\ell=\ell_0\in \mathcal{L}$; let $\ell_1, \dots, \ell_{d-1}$ be the
$d-1$ siblings of the leaf $l$. There are $d+1$ complementary
components to $\bigcup \ell_i$ in $\cdisk$. The component $C(\ell)$
bounded by $\ell$ and its $d-1$ siblings is
called the \emph{Central Strip (defined $\ell$)}.
The remaining $d$ components are $\mathcal{R}(\ell_i), 0\le i\le d-1$.

	


Lemma \ref{oscillating} is based on Lemma \ref{length:function}; the
proof is left to the reader.

\begin{lem}\label{oscillating}
Let $\mathcal{L}$ be a lamination, $\ell \in \mathcal{L}$ be a leaf,
$M=M(\lam)$ and $m=m(\lam)$. Then the length of images of $\ell$
increases $d$-fold on each step, and $H(\si^i_d(\ell))$ maps onto
$H(\si_d^{i+1}(\ell))$ homeomorphically until $\frac{|m|}d\le
|\si_d^N(\ell)|<|m|$. Then $H(\si^N_d(\ell))$ homeomorphically maps
onto $H(\si_d^{N+1}(\ell))$, and $|m|\le |\si_d^j(\ell)|\le |M|$ for
any $j>N$. Any leaf $\ell'$ with $|M|\ge \ell'\ge |m|$ is located in
$C(m)\sm C(M)$; in particular, if $i>N$ then $|M|\ge \si_d^i(\ell) \ge
|m|$, the leaf $\si_d^i(\ell)$ is located in $C(m)\sm C(M)$, and
$\si_d^i(\ell)$ never separates $\ell$ and $m$.
\end{lem}

Define a \emph{wandering triangle} as a triangle which has triangular
images under $\si_d$ so that the circular orientation on its vertices
is preserved and, moreover, has images with edges that do not cross. We
want to show that a gap $G$ of $\mathcal{L}$ is either (pre)periodic or
(pre)critical. By way of contradiction suppose $G$ is a gap which is
neither (pre)critical or (pre)periodic. Then any three vertices of $G$
form a convex hull which is, evidently also neither (pre)periodic nor
(pre)critical. Then the problem of proving that any gap $G$ of
$\mathcal{L}$ is either (pre)periodic or (pre)critical boils down to
the next lemma.


\begin{lem}\label{triangles:wander}
There are no wandering triangles whose images do not cross $\Cc(\lam)$. In particular, any gap of $\lam$ is (pre)periodic or (pre)critical, and if $\Cc(\lam)$ is (pre)periodic then any gap of $\lam$ is (pre)periodic.
\end{lem}

\begin{proof}
Let $A^0=\ch(x_0, y_0, z_0)$ be a wandering triangle whose images do
not cross $\Cc(\lam)$. Set $A^i=\si_d(A^0)$, $\sigma_d^i(x) = x_i$,
$\sigma_d^i(y) = y_i$, and $\sigma_d^i(z) = z_i$. Let $L^i$ be the
\emph{length of the longest side} of $A^i$ and let $s^i$ be the
\emph{length of the shortest side} of $A^i$ (in defining $L^i$ and
$s^i$ we use the above introduced length function so that $L^i\ge
\frac{1}{d+1}\ge s^i$ while equalities are only possible if $A^i$ is an
equilateral triangle). Evidently, there exists a continuous increasing
function $\Delta:[0, \frac{1}{d+1}]\to \R_{\ge 0}$ such that $\Delta(0)=0$
and for every $0\le t\le \frac{1}{d+1} $ the area of a triangle inscribed in
the unit disk with the shortest side at least $t$ equals $\Delta(t)$.
Denote the Lebesgue measure of a set $T$ by $\la(T)$.

Since $\la(\cdisk)<\infty$ and $\la(A^i)>0$, then $\la(A^i)$ $\conv 0$.
Since $\la(A^i) \ge \Delta(s^i)>0 $, then $s^i\conv 0$. We can
inductively choose a subsequence $i_j$ such that $i_{j+1}>i_j$ is the
first moment when $s^{i_{j+1}}<s^{i_j}$. Then, at the preceding moment,
the triangle $A^{i_{j+1}-1}$ has a side close to $\frac{1}{d}$ (this
side is the one whose image has length $s^{j_i}$). Hence the critical
set $\Cc(\lam)$ is a regular polygon with critical edges of length
$\frac{1}{d}$. Since $s^i\rightarrow 0$, the triangles $A^{i_{j+1}-1}$
have one short side and two long sides close to an edge of $\Cc(\lam)$,
of whom the closer to the edge side is of length $L^{i_j}\conv
\frac{1}{d}$.

We may assume that $i_1=2$ and $A^1$ already is close to an edge of
$\Cc(\lam)$ and has the above described properties. Denote the sides of
$A^1$ by $X^1,$ $Y^1$ and $Z^1$ where $|Z^1|\ll |Y^1|\lessapprox
|X^1|\lessapprox \frac{1}{d}$. Evidently, $s^{i_1}=d(\frac{1}{d}-|X^1|)$.
The triangle $A^{i_2-1}$ must lie in the Central Strip $C(X^1)$ of
$X^1$ (because $s^{i_2}<s^2=s^{i_1}$), thus, $A^{i_2-1}$ is squeezed between
$\Cc(\lam)$ and the $d$ siblings of $X^1$. However this implies that it
has at least one side of length less than
$\frac{1}{d}-|X^1|<s^1=d(\frac{1}{d}-|X^1|)$, a contradiction with the
choice of sequence $i_j$ (by definition, the first moment the shortest
side of the image of $A^1$ drops below $s^1$ must be $i_2$, not
$i_2-1$).
\end{proof}

Now we show that if $G$ is a non-critical periodic gap of
$\mathcal{L}$, then $G$ is finite and the first return map $g$ of $G$
permutes the sides of $G$ transitively as a rational rotation. For a
set $T$ let $\itr(T)$ be its interior.

\begin{lem}\label{transitivity}
Let $G$ be a periodic gap of $\mathcal{L}$,
and there are no critical gaps in the orbit of $G$. Then $G$ is
finite and if $g$ is the first return map of $G$, then $g$ permutes the
sides of $G$ transitively as a rational rotation.
\end{lem}

\begin{proof}

By the assumption and because $\lam$ is unicritical, the first return
map $g:\bd(G)\to \bd(G)$ is a homeomorphism that acts as a bijection on
the sides of $G$. Suppose that there are $N\ge 2$ cycles $\Aa_1, \dots,
\Aa_N$ of sides of $G$. For each $i$, choose the leaf $\ell_i\in \Aa_i$
of shortest length $s_i$; by the properties of the length function,
$s_i<\frac{1}{d+1}$. Let $\ell'_i\in \Aa_i$ be the leaf with
$\si_d(\ell')=\ell_i$. Then $L_i=|\ell'_i|>\frac{1}{d+1}$.

Denote the circle arc of length $s_i$ with the endpoints of $\ell_i$ by
$H_i$. By the above $\si_d(\Cc(\lam))\cap \uc\subset H_i$ for any $i$,
and we may number these arcs so that $\{H_1\supset \dots \supset H_N\}$
is a nested collection. Let $N\ge 3$. Then $\ell_N$ and $\ell_1$ must
be edges of distinct images of $G$. However then some image of $\ell_1$
is a leaf with endpoints contained in $H_2$, a contradiction with the
choice of $\ell_1$. Now let $N=2$. Let $\ell_i$ be an edge of $G_i$
where $G_i$ is some image of $G$, $i=1, 2$. Since $G$ is a gap, each
$\Aa_i$ has at least two leaves in $G$. If $\ell_1$ separates
$\Cc(\lam)$ from $\itr(G_1)$ then this implies that some image of
$\ell_1$ distinct from $\ell_1$ will have endpoints in $H_1$, a
contradiction with the choice of $\ell_1$. Hence the opposite takes
place and it is $\ell_1$ that is separated from $\Cc(\lam)$ by
$\itr(G_1)$. Now the fact that $H_1\supset H_2$ implies that $G_2\ne
G_1$ has all its vertices in $H_1$ so that, again, some image of
$\ell_1$ distinct from $\ell_1$, has endpoints in $H_1$, a
contradiction.
\end{proof}

A crucial Thurston's result in the \emph{quadratic case} was the
\emph{Central Strip Lemma} stating that no forward image of the
\emph{minor} of a \emph{quadratic lamination} $\mathcal{L}$ can enter
the \emph{Central Strip defined by the majors of the lamination
$\mathcal{L}$}. A version of this result holds for
\emph{unicritical laminations of arbitrary degree $d$}. Recall that the
Central Strip $C(\ell)$ of a leaf $\ell$ is closed. Call components of
$C(\ell)\cap \uc$ \emph{holes} of $C(\ell)$, and call their closures
\emph{closed holes} of $C(\ell)$. Extending this concept, for a compact set $A$ call a component of $\uc\sm A$ a \emph{hole of $A$}; if $\ell$ is a chord connecting the endpoints of a hole of $A$, we say that the hole is located \emph{behind} $\ell$.

\begin{lem}[Central Strip Lemma]\label{Central:strip}
		
Let $\ell\in \mathcal{L}, |\ell|<\frac{1}{d}$. If $\si_d^j(\ell)$ is
the first time when $\ell$ re-enters $C(\ell)$ then there are three
cases.

\begin{enumerate}

\item If $|\ell|\le \frac{1}{d(d+1)}$ then $\si_d(\ell)\subset
    C(\ell)$ and the endpoints of $\si_d(\ell)$ may belong to one or
    two closed holes of $C(\ell)$.

\item If $\frac{1}{d(d+1)}<|\ell|\le \frac{1}{d+1}$ then
    $\si_d(\ell)\subset C(\ell)$ and the endpoints of $\si_d(\ell)$
    belong to two consecutive closed holes of $C(\ell)$.

\item If $\frac{1}{d+1}<|\ell|$ then $j>1$ and the endpoints of
    $\si_d^j(\ell)$ belong to two consecutive closed holes of
    $C(\ell)$.

\end{enumerate}

\end{lem}
	
\begin{proof} Set $|\ell|=t$. If $x<\frac{1}{d+1}$, then
$|\si_d(\ell)|=\psi(t)>t$ and $\si_d(\ell)$ must be contained in
$C(\ell)$. The length of each hole of $C(\ell)$ is $\frac{1}{d}-t$.
Computations show that if $t\le \frac{1}{d(d+1)}$ then $\psi(t)\le
\frac{1}{d}-t$. Therefore in this case the endpoints of $\si_d(\ell)$ may
belong to the same closed hole of $C(\ell)$. On the other hand, it is
easy to see that if $\frac{1}{d(d+1)}<t\le \frac{1}{d+1}$ then $t\le
\psi(t)$ and $\frac{1}{d}-t<\psi(x)$ which implies the desired.

It remains to consider the case when $\frac{1}{d+1}<t<\frac{1}{d}$
($t\ne \frac{1}{d}$ by the assumptions of the lemma).
Clearly, $\psi^i(t)<t$ for $i=1, \dots, j-1$. We
claim that $\psi^i(t)>\frac{1}{d}-t$ for $i=1, \dots, j-1$. Indeed, if
$i<j$ is the least with $\psi^i(t)\le \frac{1}{d}-t$ then the
properties of $\psi$ imply $\psi^{i-1}(t)>t$, a contradiction with
$\psi^i(t)<t$ for $i=1, \dots, j-1$. Thus,
$t>\psi^{j-1}(t)>\frac{1}{d}-t$, and, hence, $\psi^j(t)>1-dt$. Since
each component of $\uc\cap C(\ell)$ is of length $\frac{1}{d}-t<1-dt$,
then $\si_d^j(\ell)$ connects distinct closed holes of $C(\ell)$ and
$|\si_d^j(\ell)|\ge t$.
Observe now that if a leaf connects two closed holes of
$C(\ell)$ that are not consecutive then its rotation by the angle
$\frac{1}{d}$ will cross itself, a contradiction.
\end{proof}

Lemma \ref{Central:strip} has important consequences.

\begin{cor}\label{Corr:CentralStrip}
Let $\mathcal{L}$ be a $\si_d$-invariant unicritical lamination. Let
$M_1, \dots M_d$ be the majors of $\mathcal{L}$ and $m=\si_d(M_1)$ be
its minor. Let $C(\lam)$ be the central strip defined by the majors of
the lamination $\mathcal{L}$. Then no forward iterate of $m$ can enter
$C(\lam)$ unless it equals $M_s$ for some $s$. 	
\end{cor}

\begin{proof}
If $\sigma_d^j(m)=\sigma_d^{j+1}(M_i)\subset C(\lam)$ then by Lemma
\ref{Central:strip} $|\sigma_d^{j+1}(M_i)|\ge |M_i|$ which by
definition of a major implies that $\sigma_d^{j+1}(M_i)=M_s$ for any
$i$ and some (unique) $s$.
\end{proof}

A $d$-gon $G$ is said to be \emph{all-critical} if $\si_d(G)$ is a
point.

\begin{thm}\label{minors:dontcross}
Minors of distinct laminations do not cross.
\end{thm}

\begin{proof}
Let $\mathcal{L}_1\ne \mathcal{L}_2$ be two laminations with
\emph{minors} $m_1$ and $m_2$. We claim that $m_1$ and $m_2$ do not
cross. Indeed, suppose otherwise, and choose a non-periodic point $x\in
\uc$ located behind both $m_1$ and $m_2$. The $\si_d$-preimage of the
point $x$ is an all-critical $d$-gon $D_x$. By Theorem \ref{prop:sat},
there exists an invariant lamination $\mathcal{L}'$ containing $D_x$.
By Corollary \ref{Corr:CentralStrip} and since $x$ is not periodic,
$\mathcal{L}'$ is compatible with $m_1$ and $m_2$ so that $m_1$ and
$m_2$ are diagonals of a gap $G_0$ of $\mathcal{L}'$. By Lemma
\ref{triangles:wander}, $G_0$ is a either pre-critical or pre-periodic.
It cannot be pre-critical because $m_1$ and $m_2$ cannot enter their
respective central strips. Thus, $G_0$ is pre-periodic. Choose $j\in
\mathbb{N}$ such that $G = \sigma_d^j(G_0) $ is periodic and
$\sigma_d^j(m_1)$ and $\sigma_d^j(m_2)$ are diagonals of $G$. By Lemma
\ref{transitivity}, the gap $G$ is finite and the first return map
permutes the sides of $G$ transitively. Hence, some forward image of
$m_1$ under $\sigma_d$ will intersect $m_1$ and some forward image of
$m_2$ under $\sigma_d$ will intersect $m_2$, a contradiction.
\end{proof}

Theorem \ref{t:uml} follows from Theorem \ref{minors:dontcross} and
Theorem \ref{t:space}.

\begin{thm}\label{t:uml}
The space of all $\si_d$-invariant unicritical laminations is compact.
The set of all their minors 
is a lamination.
\end{thm}

\begin{proof}
By Theorem \ref{t:space} if a sequence of $\si_d$-invariant unicritical
laminations converges then the limit is a $\si_d$-invariant
laminations. Since the limit of the sequence of their critical sets
then converges to a set that maps forward in the $d$-to-$1$ fashion, it
follows that the limit lamination is unicritical. This implies the
first claim of the theorem. Now, by Theorem \ref{minors:dontcross} it
is remains to prove that minors of $\si_d$-invariant unicritical
laminations form a closed family of chords. Indeed, if a sequence of
minors converges to a chord $\ell$, then we can choose a subsequence so
that the corresponding $\si_d$-invariant unicritical laminations
converge too. Their limit lamination has $\ell$ has its minor and is,
by the first claim of the theorem, unicritical. This proves the
theorem.
\end{proof}

The next definition is similar to Thurston's definition of $\qml$.

\begin{defn}\label{d:uml}
The set of all chords in the $\disk$ which are minors of some
$\si_d$-invariant unicritical lamination is a lamination called the
\emph{Unicritical Minor Lamination of degree $d$} and denoted by
$\uml_d$.
\end{defn}

\section{Basic properties of $\uml_d$}\label{Properties:umld}

We begin by studying the criterion under which a non-degenerate chord
$m$ drawn in $\disk$  or a point on $\uc$ can be the minor of some
unicritical lamination. By Lemma \ref{length:function} we know that
$|m|\le \frac{1}{d+1}$; this is a necessary condition for $m$ being a
minor. However it is not sufficient (e.g., $m$ may have forward images
that cross). Thus, to establish the desired criterion we need more
conditions on iterations of $m$.
Consider pullbacks of $m$ of length at most $\frac{1}{d}$ and call them
\emph{`short' pullbacks}. `Short' pullbacks of a chord may well be
longer than the chord itself. There are $2d$ `short' pullbacks of $m$;
they are edges of a well-defined $2d$-gon. Observe that if $|m|\le
\frac{1}{d+1}$, then there are $d$ `short' pullbacks of $m$ that are of
length $\frac{|m|}d<|m|$ and $d$ `short' pullbacks of $m$ that are of
length $\frac{1-|m|}d\ge |m|$ (in the latter case equality means that
$|m|=\frac{1}{d+1}$). This observation follows from the properties of
the length function $\psi$ described earlier. Given a chord $m=\ol{ab}$
with $|m|\le \frac{1}{d+1}$, set $D(m)$ to be the convex hull of
$\si_d^{-1}(\{a, b\})$.

\begin{lem}\label{condition:minor}
The following conditions form a criterion for a non-degenerate leaf $m$
to be a minor of unicritical lamination:

\begin{enumerate}

\item $|m|\le \frac{1}{d+1}$;

\item $m$ and its forward images do not cross;

\item no image of $m$ is shorter than $m$;

\item If $M_1$,$\dots$, $M_d$ are $d$ `short' pullbacks of $m$ of
    length greater  than $\frac{1}{d+1}$, then forward images of
    $m$ do not cross $M_1$,$\dots$, $M_d$.

\end{enumerate}

Moreover, Thurston's pullback lamination $\lam(m)$ based on $m$ and
$M_1, \dots, M_d$, has $m$ as its minor.
\end{lem}

\begin{proof}
First we show that if $m$ is the minor of a unicritical lamination
$\lam$, then it has the listed properties. Indeed, by Lemma
\ref{length:function}, $|m|\le \frac{1}{d+1}$. Moreover, by properties
of laminations, $m$ and its forward images do not cross. Consider $d$
`short' pullbacks $M_1$, $M_2$,$\dots$, $M_d$ of $m$ of length greater
than $\frac{1}{d+1}$. Then, again by properties of laminations, forward
images of $m$ do not cross $M_1$, $M_2$,$\dots$, $M_d$. Finally, no
image of $m$ is shorter than $m$ by Corollary \ref{Corr:CentralStrip}.

On the other hand, we claim that if $m$ has the listed properties then
$m$ is the minor  of some unicritical lamination. Indeed, consider the
$2d$-gon $D(m)$; it has $d$ edges $M_1, M_2, \dots, M_d$ of length
$\lambda=\frac{1-|m|}d\ge \frac{1}{d+1}$ and $d$ edges $M'_1, M'_2,
\dots, M'_d$ of length $\frac{|m|}d<|m|$. The leaf $m$ cuts $\disk$
into two pieces: a small piece $T_0$ and a large piece $T_1$. The
$2d$-gon $D(m)$ cuts $\uc$ in $2d$ arcs: holes $E_1, \dots, E_d$ of $T$
behind $M_1, \dots, M_d$, and holes $E'_1, \dots, E'_d$ behind $M'_1,
\dots, M'_d$. Clearly, this entire picture is symmetric with respect to
the rotation by $\frac{1}d$. Observe that $E_1, \dots, E_d$ map in the
homeomorphic fashion to $T_1$ by $\si_d$ while $E'_1, \dots, E'_d$ map
by $\si_d$ to $T_0$ in the homeomorphic fashion.

We claim that chords $\si^j_d(m)$ never enter sets $E'_i$. Indeed,
otherwise the assumption  that images of $m$ do not cross chords $M_1,
\dots, M_d$ implies that for some $j$ both endpoints of $\si^j_d(m)$
belong to sets $E'_1, \dots, E'_d$ and hence $\si_d^{j+1}(m)\subset
T_0$, a contradiction with the assumption that no image of $m$ is
shorter than $m$. Thus, the forward orbit of $m$ together with $T$
forms a forward invariant lamination $\mathcal{L}_0$. Using Thurston's
pullback construction and relying upon the unicriticality, we now
construct a desired invariant lamination $\mathcal{L}$ (once we know
that $\mathcal{L}_0$ is forward invariant the pullback construction
always goes through).
\end{proof}

\begin{defn}\label{d:lamm}

If a non-degenerate chord  $m$ satisfies the above condition from Lemma
\ref{condition:minor},  then we will call the lamination
$\mathcal{L}(m)$ constructed in Lemma \ref{condition:minor} and having
$m$ as the minor as the \emph{lamination generated by $m$} or the
\emph{lamination corresponding to the minor $m$}. The notation
$\mathcal{L}(m)$ will be used from now on for that lamination.

\end{defn}

\begin{lem}\label{l:partialorder}
The following claims hold.

\begin{enumerate}

 \item If $n\in \uml_d$, $m\in \lam(n)$ is a leaf, $n\succ m$, and
     there exists no $i$ such that $n\succ \si_d^i(m)\succ m$ then
$m\in \uml_d$.

 \item If $m, n\in \uml_d$ with $m\succ n$, then $m$ is a leaf of
     $\mathcal{L}(n)$.

\end{enumerate}
\end{lem}

\begin{proof}
(1) To prove that $m \in \uml_d$, we show that $m$ satisfies the three
conditions of Lemma \ref{condition:minor} one by one. Under iteration
of $\si_d$, the length of $m$ will continue to increase until it is at
least $\frac{1}{d+1}$ and after that by Lemma \ref{oscillating} the
image leaves cannot become shorter than $n$. Since, all leaves on the
forward orbit of $m$ are contained in $\mathcal{L}(n)$, the first two of the
four hypothesis of Lemma \ref{condition:minor} are satisfied.

Let $M_1,M_2,\dots M_d$ be $d$ leaves of length at least
$\frac{1}{d+1}$ that are  immediate preimages of $m$( these would be
the major leaves of the lamination corresponding to $m$). They are
longer than the major leaves of $\mathcal{L}(n)$ and thus not leaves of
$\mathcal{L}(n)$. In order to be able to apply Lemma
\ref{condition:minor} we need to show that no leaf in $\mathcal{L}(n)$
can intersect the interiors of the leaves $M_1,M_2,\dots M_d$. Since,
no leaf in $\mathcal{L}(n)$ could intersect the major leaves of
$\mathcal{L}(n)$, a leaf in $\mathcal{L}(n)$ could intersect
$M_1,M_2,\dots M_d$ only if its length is at most $\frac{1}{d}$ times
the length of $n$.

If a leaf $m'$ in the forward orbit of $m$ is longer than
$\frac{1}{d+1}$  (but shorter than the major leaves of
$\mathcal{L}(n)$), the image leaf is shorter than $m'$, but at
least as long as $n$ by Lemma \ref{oscillating}; from then on, no leaf
in the orbit of $m'$ can intersect the interior of $M_1,M_2,\dots M_d$.
Thus, a leaf in the orbit of $m$ can intersect the interior of $M_1,M_2,\dots
M_d$ only during the initial iterations while the length
of the leaf increases by a factor of $d$. If this does not
happen, then by Lemma \ref{condition:minor}, $m$ is the minor of an
invariant lamination and thus contained in $\uml_d$. Clearly, $n\succ
m$.

(2) We will prove the claim for the case $n$ is non-degenerate. The
proof for the degenerate case is similar. We first show that $m$ does
not intersect any leaf in the lamination  $ \mathcal{L}(n)$. If
$\mathcal{C}(n)$ and $\mathcal{C}(m)$ are the central strips defined by
the majors of the laminations $\mathcal{L}(n)$ and $ \mathcal{L}(m)$
respectively  then $\mathcal{C}(n)\subset \mathcal{C}(m)$. By Lemma
\ref{Central:strip} no image of $m$ crosses a major of
$\mathcal{L}(n)$. By construction, no image of $m$ crosses a leaf of
$\mathcal{L}(n)$. 	

We claim that $m$ cannot divide a gap in $\mathcal{L}(n)$. Since, the
minor  $n$ is non-degenerate, so $\mathcal{L}(n)$ has a critical
$2d$-gon $G$. Each gap in $\mathcal{L}(n)$ is either an eventually
periodic polygon or a gap which eventually maps to $G$ (this follows
from the absence of wandering triangles established in Lemma
\ref{triangles:wander}). However, by Lemma \ref{Central:strip} no image
of $m$ can divide $G$. On the other hand, if $m$ divides a pre-periodic
polygon, then, by Lemma \ref{transitivity}, two forward images of $m$
will intersect which is a contradiction. Thus, $m$ cannot divide a gap
in $\mathcal{L}(n)$. Therefore, $m$ is a leaf in $\mathcal{L}(n)$.
\end{proof}

Given a non-diameter chord $\ell$, let $A(\ell)$ be the set (possibly empty) of all points that never exit $\ol{H(\ell)}$. Given an arc $I$ we denote by $|I|$ its length
and that we normalize the circle length so that it equals $1$.

\begin{lem}\label{l:fxpt}
Suppose that $\ell=\ol{yz}$ is a chord of length at most $\frac{1}{d}$ such that a fixed point $a$ belongs to $[y, z]$. Then the following holds:

\begin{enumerate}

\item  The only way $\si_d(y)$ and $\si_d(z)$ can belong to $[y, z]$ is when one of these points equals $a$ and $|\ell|=\frac{1}{d}$.

\item The unique minimal invariant set contained in $[y, z]$ is $\{a\}$.

\item If $m$ is a minor with a fixed endpoint $a$ then $m=\{a\}$.

\end{enumerate}

\end{lem}

\begin{proof}
For $0<[a, z]=\lambda<\frac{1}{d}$ we have
$|[a, \si_d(z)]|=d\lambda<1-(\frac{1}{d}-\lambda)<1-|[y, a]|$
which implies that unless one of points $y, z$ equals $a$ and $|\ell|=\frac{1}{d}$, all points of $[y, z]$ except $a$ are eventually mapped outside $[y, z]$. This shows that,  except in the case when $|\ell|=\frac{1}{d}$ and either $y$ or $z$ equals $a$, the set $A(\ell)$ equals $\{a\}$ and claims (1) and (2) are immediate. In the remaining case (1) clearly holds, and $A(\ell)$ consists of $a$ and a countable family of pullbacks of the other endpoint of $\ell$ towards $a$. Still, it follows that (2) holds in this case too. To prove (3),
let $m=\ol{ab}$ be the non-degenerate minor of a lamination $\lam$.
Choose $b'\in H(m)$ with $\si_d(b')=b$. Then the major of $\lam$ with
an endpoint $b'$ must cross $m$ or be located ``under'' $m$, a
contradiction.
\end{proof}

If $m$ is the minor of a lamination $\lam$ which is not a fixed point,
then we denote by $M_m$ the major of $\lam$ such that $m\subset
\ol{H(M_m)}$ (if $m$ is a fixed point then $M_m$ is not well-defined as there are two majors with the above property).

\begin{lem}\label{l:minor-fxpt}
Let $M=\ol{yz}$ be a chord with $|M|\le \frac{1}{d}$ such that $M\succ \si_d(M)=m$. Then there are two cases.

\begin{enumerate}

\item $M$ is a chord of length $\frac{1}{d}$ with a fixed endpoint.

\item Otherwise $y\le \si_d(z)\le \si_d(y)\le z$, $\si_d([y, z])=[\si_d(y), \si_d(z)]$,
$\si_d(\ell)\in [y, z]$ for any critical chord $\ell\succ M$ with
$|\ell|=\frac{1}{d}$, and the arc $[z-\frac{1}{d}, y+\frac{1}{d}]=I$
contains no fixed points (so, $\ol{H(M)}$ and $\ol{H(m)}$
contain no fixed points).
\end{enumerate}

In particular, this holds if $m$ is a minor and $M=M_m$.
\end{lem}

\begin{proof}
(1) For a fixed point in $\ol{H(M)}$ the claim follows by Lemma \ref{l:fxpt}.

(2) Suppose that no fixed point belongs to $\ol{H(M)}$. Then
the first two claims follow from the fact that the image of $H(M)$ is an
arc of length at most $1$ but greater than $|H(M)|$.  Now, since
$\si_d$ preserves orientation we see that as we rotate a critical chord
$\ell$ from $\ol{y (y+\frac{1}{d})}$ to $\ol{(z-\frac{1}{d}) z}$, the
point $\si_d(\ell)$ moves along the circle in the negative direction
from $\si_d(y)$ to $\si_d(z)$ which proves the third claim. To prove
the last claim, observe that if a fixed point $a$ belongs to
$[z-\frac{1}{d}, y+\frac{1}{d}]$, then there exists a critical chord
$\ell_a$ of length $\frac{1}{d}$ with both endpoints in $I$ such that
$\ell\succ M$. Since by the second claim $\si_d(\ell_a)\in [y, z]$,
then $a$ coincides with one of the endpoints of $M$, a contradiction.
\end{proof}

For a point $x\in \uc$, let $D_x$ be the convex hull of the full
preimage of $x$ under $\si_d$; clearly, $D_x$ is an all-critical
$d$-gon.

\begin{lem}\label{degenerate:point}
Every $x\in \uc$ is the minor of a unicritical lamination.
\end{lem}

\begin{proof}
The set $D_x\cup \uc$ is forward invariant, has $d$ critical leaves,
and, hence, satisfies Thurston's pullback construction conditions
stated in Theorem \ref{prop:sat}.  Thus, there exists a
$\si_d$-invariant lamination having minor $x$. Such a lamination is
clearly unicritical.
\end{proof}

Denote the lamination from Lemma \ref{degenerate:point} by $\lam_x$.
By our construction both $\lam(m)$ (if $m$ is a
non-degenerate minor) and $\lam_x$ (if $x\in \uc$ is a point) are
sibling invariant laminations. By Lemma \ref{degenerate:point}, a point $x \in \uc$ may be either an
endpoint of a non-degenerate minor or there may exist no such
non-degenerate minor in which case $x$ is called a \emph{degenerate minor} (i.e., it is a point of the circle disjoint from other minors); by Lemma \ref{degenerate:point} every point of the circle is a minor, but only points disjoint from other minors will be called \emph{degenerate minors}. We will now
investigate these phenomena. If $x$ is periodic,  one vertex of the
$d$-gon $D_x$ is periodic,  the others are not. The simplest case is
when $x$ is a fixed point.

\begin{lem}\label{l:fxpt-lami}
If $x$ is a fixed point, $\lam_x$ consists of isolated leaves that
accumulate on points of $\uc$. All gaps of $\lam_x$ are pullbacks of
$D_x$. If $\lam$ is a unicritical lamination such that $x$ belongs to the image of its critical set $\Cc(\lam)$, then $\lam=\lam_x$. Thus, $x$ is a degenerate minor.
\end{lem}

\begin{proof}
The gap $D_x$ has $d$ holes. Under the map $\si_d$ they expand to cover
the circle exactly once. Moreover, for each such hole $I$ there exists
a unique fixed point $y_I\in \ol{I}$. Denote the edge of $D_x$
connecting the endpoints of $I$ by $\ell_I$. Then the $\si_d$ pullback
of $D_x$ with vertices in $I$ is a $(d+1)$-gon attached to $D_x$ at
$\ell_I$ with vertices that include the endpoints of $I$ and otherwise
partition $I$ into $d$ equal arcs of length $\frac{1}{d^2}$ each. On
the next step we will obtain new pullbacks of $D_x$ with vertices that
partition each of those arcs of length $\frac{1}{d^2}$ into yet $d$
equal arcs, etc. Passing to the limit we arrive at the first claim of the
lemma.

Now, let $\lam$ be a unicritical lamination such that $x$ belongs to the image of its critical set $\Cc(\lam)=\Cc$. Then $D_x\subset \Cc$, and all leaves of $\lam$ are compatible with $\lam_x$ (i.e., do not cross leaves of $\lam_x$). Suppose that $\lam$ has a leaf $\ell\notin \lam_x$. Then $\ell$ is contained in one of the pullbacks of $D_x$ which means that an eventual image of $\ell$ is contained in $D_x$ but is not an edge of $D_x$. Evidently it is impossible because $\lam$ is unicritical.
\end{proof}

Lemma \ref{nontrivial:object} is similar except now $x$ is periodic of
period $p>1$.

\begin{lem}\label{nontrivial:object}
Let $x$ be a non-fixed point. Then:

\begin{enumerate}

\item a unique hole $H_1$ of $D_x$ contains no $\si_d$-fixed points
    while each other hole of $D_x$ contains exactly one
    $\si_d$-fixed point;

\item $H_1$ contains a non-degenerate minimal set $\mathcal{I}_x$;


\item if $x$ is periodic of period $p>1$, and $\mathcal{L}_c$ is
    obtained from $\mathcal{L}_x$ by removing $D_x$ and its
    backward orbit, then $\lam_c$ is non-empty, and there is a
    Fatou gap $U$ of $\mathcal{L}_c$ of period $p$ that contains
    $D_x$.

\end{enumerate}

\end{lem}

\begin{proof}
(1) Clearly, $\{x\}$ is the minor of $\lam_x$ and is not a vertex of
$D_x$. Choose a hole $H_1$ of $D_x$ such that $x\in H_1$. Then by Lemma
\ref{l:minor-fxpt}, no fixed point belongs to $H_1$. Since the arcs
between two consecutive fixed points are of length
$\frac{1}{d-1}>\frac1d$, $H_1$ is a \emph{unique} such hole of $D_x$.

(2) By (1) the set of points that stay forever in $\ol{H_1}$ is an
invariant compact set that contains no $\si_d$-fixed points. Hence,
there exists the desired minimal set $\mathcal{I}_x$.

(3) $\mathcal{L}_c$ is a lamination: all properties of laminations are
immediate for $\mathcal{L}_c$ except for the fact that $\lam_c$ is
closed, and to see that, observe that by construction all edges of
$D_x$ are isolated in $\lam_x$; hence, all pullbacks of edges of $D_x$
are isolated in $\lam_x$ and $\mathcal{L}_c$ is closed.

Since the period of $x$ is $p$, the $(p-1)$-th pullback of $D_x$ is a
$d$-gon  $\Delta_x$ emanating out of $x$. Taking $\sigma_d$-preimage of
$\Delta_x$ we get $d$ polygons $\Gamma^1\subset \ol{H_1},$ $\dots,$
$\Gamma^{d}\subset \ol{H_d}$ each of which is a $d+1$-gon called an
\emph{immediate decoration}, hanging out from an edge of $D_x$. Further
pullbacks of $D_x$ hang out of edges of immediate decorations (these
are called simply \emph{decorations}), etc.

Consider the convex hull $G$ of $\mathcal{I}_x$. We claim that
pullbacks of  $D_x$ can at most ``touch'' $G$ (i.e., have common
vertices with $G$), but otherwise are disjoint from $G$. Indeed, edges
of $G$ cannot intersect pullbacks of edges of $D_x$ as otherwise a
forward image of an edge of $G$ will cross an edge of $D_x$, a
contradiction. Thus, the only way the claim of the lemma fails is if
$G$ contains a pullback of an edge of $D_x$. Since $\mathcal{I}_x$ is
invariant, then an edge of $D_x$ contained in the boundary of $H_1$, is
an edge of $G$. Since $\mathcal{I}_x$ is minimal, it follows then that
$\mathcal{I}_x$ coincides with the periodic orbit of $\si_d(x)$.
Evidently, this contradicts the assumption that an edge of $G$ is also
an edge of $D_x$. The claim that the intersection of a pullback of
$D_x$ and $G$ is at most a common vertex, follows.

We claim that $G$ is contained in a non-degenerate invariant gap of $\lam_c$.
Indeed, by the above $G$ is contained in a non-degenerate gap $G'$ of $\lam_x$;
edges of $G'$ are not pullbacks of edges of $D_x$ as at each edge of a pullback
of $D_x$ yet another pullback of $D_x$ is attached. Hence all edges of $G'$ are
limits of pullbacks of  edges of $D_x$, 
and $G'$ remains a non-degenerate gap of $\lam_c$, and $\lam_c$ is not
an empty  lamination. Since $\lam_c$ is obtained from $\mathcal{L}_x$
by removing $D_x$ and its backward orbits, one can say that pullbacks
of $D_x$ do not fill up the entire unit disk.

The gap $U$ of $\mathcal{L}_c$ that contains $D_x$ is an infinite
periodic gap  which returns with degree $d$ (it is infinite because by
construction it contains immediate decorations as well as other
decorations consecutively attached to them). Hence, $U$ is a Fatou gap.
Let the period  of $U$ be $k$. If $k<p$, then $\sigma_d^k$ brings the
gap $U$ back to itself but the point $\tilde{x}$ does not return to its
initial position and has is of period $\frac{p}{k}>1$ inside $U$ under
the first return map $g=\sigma_d^k$. So, using the same technique as
above applied to $\si_d^k|_U$, one can show that the gap $U$ contains a
non-degenerate gap of $\lam_c$, a contradiction. Thus, $k \nless p $.
On the other hand, by construction $U$ contains $D_x$ and all its
immediate decorations which implies that $\si_d^p(U)=U$. Hence $k=p$
and the claim follows.
\end{proof}

	

Lemma \ref{nontrivial:object} shows how unicritical laminations with
periodic  Fatou gaps can be constructed. We will now study such
laminations. Throughout this section we use the same notation for the
now defined objects. Let $m(\lam)=m$ be the minor of $\lam$. As in
Lemma \ref{condition:minor}, there are $d$ pullbacks of $m$ denoted by
$M_1, M_2, \dots, M_d$ that are of equal length
$\lambda=\frac{1-|m|}d\ge \frac{1}{d+1}$. Since $m$ is the minor of
$\lam$, chords $M_1, M_2, \dots, M_d$ must be leaves of $\lam$ (one of
them is, hence all of them are since $\lam$ is unicritical). Observe
that they are $1/d$-rotations of each other. By Lemma
\ref{Central:strip} forward images of $m$ do not enter the central
strip $C$ of $\lam$, i.e. the component of $\cdisk\sm
\bigcup_{i=1}^{d-1} M_i$ that has all $M_i$'s on its boundary. In
particular, no forward image of $m$ coincides with a chord connecting
two consecutive endpoints of adjacent majors of $\lam$.

Consider now the case when the minor $m(\lam)=m$ of $\lam$ is periodic
of  period, say, $n$. Since $\si_d(\si_d^{n-1}(m))=m$, then
$\si_d^{n-1}(m)$ is one of the majors $M_1, \dots, M_d$. Denote
$\si_d^{n-1}(m)$ by $M(\lam)=M$ and call it the \emph{primary major} of
$\lam$. The remaining $d-1$ majors of $\lam$ are called
\emph{secondary}. Clearly, $M$ and $m$ have the same period.

\begin{lem}\label{l:prim-refixed}
Suppose that $\lam$ has an $n$-periodic cycle of Fatou gaps with $U$
critical.  Then all $d$ majors of $\lam$ are edges of $U$ and exactly
one of them is periodic of period $n$.
\end{lem}

\begin{proof}
There are $d$ majors of $\lam$ of length at least $\frac{1}{d+1}$. By
Lemma  \ref{Central:strip} $\si_d^n$ maps them all to one of them which
is periodic of period $n$.
\end{proof}

Given a Fatou gap $U$ of period $k$, we call an edge (or vertex) of $U$
\emph{refixed} if its endpoints are of period $k$. The periodic major
of $\lam$ discovered in Lemma \ref{l:prim-refixed} is, evidently,
refixed.

\begin{lem}\label{firstreturn:dfixed}
Suppose that $\lam$ has an $n$-periodic cycle of Fatou gaps and $U$ is
a  critical gap in that cycle. Then each edge of $U$ is an eventual
preimage of $M(\lam)$. In particular, the first return map $g|_U$ has
exactly $d$ fixed points. Two of these points are the endpoints of
$M(\lam)$. Otherwise each component of $\bd(U)$ from which all majors
of $\lam$ are removed contains exactly one fixed point of
$g|_{\bd(U)}$.
\end{lem}

\begin{proof}
Evidently, $U$ has $d-1$ edges (possibly, degenerate edges, i.e.
vertices) that are refixed. In particular $M(\lam)$ is a refixed edge
of $U$. By Lemma \ref{length:function}, $|M(\lam)|\ge \frac{1}{d+1}$.
Thus, $\lam$ has $d$ majors all of which are edges of $U$ of length at
least $\frac{1}{d+1}$.

Take an edge $N$ of $U$. It is well-known that any edge of $U$ is
eventually  mapped into either critical or periodic edge of $U$. Since
there are no critical edges of $U$, we may assume that $N$ is itself
periodic. Let $\si_d^i(N)$ be the longest leaf in the orbit of $N$. By
Lemma \ref{length:function}, $|\si_d^i(N)|\ge \frac{1}{d+1}$. It
follows that there exists a major $M'$ of $\lam$ such that the
endpoints of $\si_d^i(N)$ belong to the hole of $U$ behind $M'$.
However then, as we iterate $\si_d^i(N)$, it cannot be mapped to a
chord longer than $\si_d^i(N)$ (by the choice of $\si_d^i(N)$), and
neither can it enter its central strip (by Lemma \ref{Central:strip}),
a contradiction with $N$ being an eventual image of $\si_d^i(N)$.
\end{proof}

We will now prove the first main result of this section.

\begin{thm}\label{periodicpt:2possibilities}
	For a $p$-periodic point $x \in \ucirc, p>1,$ there are precisely two
	possibilities.
	
\begin{enumerate}	
\item \textbf{Degenerate case}: $x$ is the degenerate minor of some
    unicritical  lamination but there  exists no unicritical
    lamination with a non-degenerate minor one of whose endpoints
    is $x$.

\item \textbf{Non-degenerate case}: $x$ is an endpoint of a unique
    non-degenerate minor $m$ of some unicritical lamination.
\end{enumerate}
	
\end{thm}

\begin{proof}
Suppose that the degenerate case does not hold and there exists a minor
$m$ with endpoint $x$. We claim that such $m$ is unique. Suppose
otherwise. We may assume that $m=\ol{xy}$, $y$ is period $p$, $H(m)=(x,
y)$ (the movement from $x$ to $y$ inside $(x, y)$ is clockwise) and for
some periodic point $z\in (x, y)$ of period $p$ the chord $\ol{xz}$ is
also the minor of some unicritical cubic lamination. Let the
corresponding \emph{primary majors} be $M_y=\ol{x'y'}$ and
$M_z=\ol{x'z'}$. Also, let $U_y$ be the critical Fatou gap with major
$M_y$ and $U_z$ be the critical Fatou gap with major $M_z$.

Denote by $L$ and $R$ the two siblings of $M_y$ adjacent to $M_y$.
Since by Lemma \ref{length:function} $\frac{1}{d+1}\le
|M_z|<\frac{1}{d}$, then the point $z'$ must belong to an arc $I$ that
separates (in the circle) $M_y$ from one of those two siblings $L$ and
$R$, but does not contain the other one. Since this arc contains no
$p$-periodic points that belong to $\bd(U_y)$, then there is an edge
$\ell$ of $U_y$ with endpoints in $I$ that crosses $M_z$. Recall that
by Lemma \ref{firstreturn:dfixed} $\ell$ is an eventual preimage of
$M_y$. Since $\si_d^p$ acts on $U_y$ as $\si_d$ (after we collapse all
edges of $U_d$ to points), then after several iterations of $\si_d^p$
the leaf $\ell$ will map onto a sibling $M''_y$ of $M_y$ distinct from
$M_y$; it is easy to see that, as it happens, the leaf $M_z$ will keep
crossing the corresponding images of $\ell$. In particular, the
appropriate image of $M_z$ will cross $M''_y$. However the location of
$M_z$ shows that then this image of $M_z$ will cross one of its
siblings, a contradiction.
\end{proof}

\section{The proof of the fact that $\uml_d$ is a
q-lamination}\label{s:umlq}

First let us classify all gaps of $\uml_d$. If $m$ is a minor, then by Lemma \ref{l:minor-fxpt} the arc $H(m)$ contains not
fixed points. Thus, if we connect the center $(0, 0)$ of $\disk$ with
fixed points of $\si_d$ by $d-1$ radii, then each minor is located in
one of the thus created $d-1$ circular sectors. Any gap $G$ of $\uml_d$
has a side $m$ that ``faces'' $(0, 0)$ (separates $(0, 0)$ from the
rest of $G$). Other edges $n$ of $G$ are incomparable among themselves
because their sets $H(n)$ are pairwise disjoint. Thus, all edges of $G$
are \emph{immediate successors} of $m$ in the partial order $\succ$.
Call $m$ the \emph{leading edge} of $G$.


Let $\mathcal{L}$ be a lamination with minor $m=\ol{ab}$. 
$\mathcal{L}$ may have a critical gap $\Cc(\lam)$ which is either a
pre-periodic polygon with $dn$ sides ($n\ge 3$) or a periodic Fatou gap
(this holds by Lemma \ref{triangles:wander}). Consider all three cases
for $\Cc(\lam)$.

\begin{lem}\label{polygonal:gaps}
If $\Cc(\lam)$ is a polygon with $dn$ sides($n\ge 3$), then
$\si_d(\Cc(\lam))$ is a gap of $\uml_d$.
\end{lem}

\begin{proof}
The gap $\Cc(\lam)=G$ can be subdivided by adding a collapsing $2d$-gon
that coincides with the convex hull of $d$ sibling edges of $G$ with the same image which is an edge of $\si_d(G)=G'$. That collapsing $2d$-gon, apart from $d$ sibling edges of it shares with $G$, will have $d$
complementary edges connecting the appropriate endpoints of the sibling
edges of $G$ mentioned above.
This can be done for each edge of $G'$, thus it can be done in $n$ ways. Each edge of $G'$ is the minor of the corresponding lamination obtained by pulling back the appropriate $2d$-gon inside pullbacks of $G$.
We claim that $G'$ is a gap in $\uml_d$. Indeed,
otherwise $G'$ contains a minor which is a diagonal of $G'$. However,
by Lemma \ref{triangles:wander} and Lemma \ref{transitivity} every
diagonal of $G'$ has crossing forward images and, hence, cannot be a
minor. Thus, $G'$ is an actual gap of $\uml_d$.
\end{proof}

To consider the case when $\Cc(\lam)$ is a Fatou gap we need to study
the unicritical version of the quadratic Main Cardioid. We first prove the following Lemma:

\begin{lem}\label{l:part1}
Let $\ell$ be a critical chord of length $\frac{1}{d}$. Then there
exists an 
invariant minimal set $\Ii\subset \ol{H(\ell)}$, and for each such $\Ii$
the convex hull $G$ of $\Ii$ has the longest edge $M$ that
separates the rest of $\Ii$ from the center of $\disk$; the image
$m=\si_d(M)$ is a minor, and the unicritical lamination $\lam(m)$
includes $G$ as an invariant gap or leaf.
\end{lem}

\begin{proof}
By Lemma \ref{l:fxpt} we may consider a critical chord $\ell$ of length
$\frac{1}{d}$ with no fixed point in $\ol{H(\ell)}$.  The set
$A(\ell)\ne \0$ (since $\si_d(\ol{H(\ell)})=\uc$) of all points with
orbits contained in $\ol{H(\ell)}$ is non-degenerate (by the ``no fixed
points in $\ol{H(\ell)}$'' assumption), forward invariant and onto (the
latter follows again because $\si_d(\ol{H(\ell)})=\uc$), with the
circular order preserved on $A(\ell)$ except when the endpoints of
$\ell$ belong to $A(\ell)$. Take a minimal set $\Ii\subset
\ol{H(\ell)}$. If $A(\ell)$ consists of two points, then $A(\ell)=\Ii$
is a two periodic orbit, and we associate with $A(\ell)$ the circle
rotation by $\frac12$ and the rotation number $\frac12$. If $A(\ell)$
consists of more than two points we consider $\ch(A(\ell))=U(\ell)$
with the map $\si_d|_{U(\ell)}$ (recall that we canonically extend
$\si_d$ onto chords of the unit circle). Since the circular orientation
on $A(\ell)$ is preserved, the degree of $\si_d|_{\bd(U(\ell))}$ is
$1$.

Thus, we can associate to $\si_d|_{\bd(U(\ell))}$ (and, hence, to
$\ell$) a well-defined rotation number $\rho(\ell)\in (0, 1)$ (there
are no fixed points in $A(\ell)$!). Consider the convex
hull $G$ of $\Ii$. Among edges of $G$ there is a unique edge $M$
that ``faces'' the center $(0, 0)$ of $\disk$ (separates $(0, 0)$ from
the rest of $G$). This is the longest edge of $G$. Set $m=\si_d(M)$. By
Lemma \ref{length:function}, $|M|\ge \frac{1}{d+1}$, and by Lemma
\ref{l:minor-fxpt} $\si_d(\ell)\in H(m)$. Consider two cases related to
the properties of any minimal invariant set $\Ii\subset \ol{H(\ell)}$
(later on we will prove that $\Ii$ is unique).

(1) Suppose that $\Ii$ is finite, and, hence, is a periodic orbit. This
holds if and only if $\rho(\ell)$ is rational. As we map $m$ forward,
the images of $m$ are all  oriented in the same way inside $H(M)$
because, as long as the points stay in $H(\ell)$, their circular order
is preserved. Their lengths stay under $\frac{1}{d}$, and on each step
they are simply multiplied by $d$.

This allows us to evaluate lengths of all edges of $G$ if we know the
period $q$ (not even the rotation number).

Set $\ol{ab}=m=\si_d(M)$ and consider the critical $2d$-gon $D(m)$. It
has $M$ and all its siblings as edges and in fact coincides with
$\si_d^{-1}(\{a, b\})$. Since $D(m)$ and $G$ form a forward invariant
lamination, Thurston's pull back construction goes through and gives
rise to the canonical pullback lamination $\lam(m)$ that has $m$ as its
minor. Such minors are said to be \emph{prime rational minors}.

(2) If $\Ii$ is infinite, then it is known that $G$ is an
invariant gap such that when we collapse all its edges to points and
transfer the action of $\si_d$ onto the resulting circle we get an
irrational rotation. Such gaps are called invariant \emph{Siegel} gaps.
Since edges of an invariant gap are (pre)periodic or (pre)critical,
edges of $G$ are (pre)critical. By our setup, $G$ can only have one
critical edge coinciding with $\ell$. Let us show that then no two
edges of $G$ can intersect. Indeed, if they do, then after
applying $\si_d$ the appropriate number of times we would see that both
$\ell$ and an edge $\ell'$ of $G$ have a common endpoint, say,
$b$. However $\ell'$ must eventually map to $\ell$ thus implying that
$\si_d(\ell)$ is periodic, a contradiction. Hence no two edges of
$G$ have a common endpoint, the entire set $A(\ell)$ is minimal,
and coincides with the limit set of, say, $\si_d(\ell)$. This case
corresponds to an irrational rotation number $\rho(\ell)$. The
degenerate minor $\si_2(\ell)$ is called a \emph{prime irrational
minor}.
\end{proof}

\begin{defn}\label{d:irrr}
Minors whose existence is proven in Lemma \ref{l:part1} are called
\emph{prime minors}. If $G(\ell)$ is a finite gap or non-degenerate
leaf then the corresponding minor is said to be \emph{rational}. If
$G(\ell)$ is a Siegel gap then the corresponding degenerate minor is
said to be \emph{irrational}.
\end{defn}

\begin{lem}\label{prime:incomparable}
Two distinct prime minors are not $\succ$-comparable.	
\end{lem}

\begin{proof}
Indeed, suppose that $m\succ m'$ are two prime minors. Let $M_m=M$ and
$M_{m'}=M'$ be the corresponding majors. Then by Lemma
\ref{l:minor-fxpt} $M'\succ M$. Choose a critical chord $\ell\succ M'$
of length $\frac{1}{d}$. Since the orbit of $m$ is contained in
$\ol{H(M)}$ and the entire orbit of $m'$ is contained  in $\ol{H(M')}$,
then in fact the orbits of $m$ and $m'$ are contained in $\ol{H(\ell)}$
which implies that both minors are rational and have the same period.
However then by the above they must have the same length and, hence,
must coincide.
\end{proof}

\begin{lem}\label{invariant:unique}
The minimal set $\Ii$ in Lemma \ref{l:part1} is unique	
\end{lem}

\begin{proof}
Notice that if case (2) of Lemma \ref{l:part1} holds, then the claim is
immediate. So, assume that $\rho(\ell)$ is rational and show that $\Ii$
is well-defined. Suppose otherwise. Then there are two prime rational
minors $m$ and $m'$ generated by minimal invariant sets $\Ii\subset
\ol{H(\ell)}$ and $\Ii'\subset \ol{H(\ell)}$. Recall that by Theorem
\ref{minors:dontcross} $m$ and $m'$ do not cross, but might have a
common endpoint.
We claim that $m$ and $m'$ are $\succ$-comparable. Indeed, for geometric
reasons $\si_d(\ell)\in \ol{H(m)}\cap \ol{H(m')}$. It follows that the only way
$m$ and $m'$ are not $\succ$-comparable is when $\si_d(\ell)=y$ is the
terminal point of one of them and the initial point of the other one.
However then $y\in \Ii\cap \Ii'$, and since these sets are minimal, they coincide, 
a contradiction. Hence, $m$ and $m'$ are $\succ$-comparable and disjoint. We may assume that
$m=\ol{xy}, m'=\ol{x'y'}$ are of period $n$, and $x<x'<y'<y$. The images of the arc $H(m)$ map onto each
other and stay shorter than $\frac{1}{d}$ until $m$ maps to the corresponding major $M$
and $m'$ maps to the corresponding major $M'$. However then the $\si_d^n$-image of
$[x, x']$ stays shorter than $\frac{1}{d}$, and from that moment on the whole orbit of 
$[x, x']$ gets repeated. However, $[x, x']$ must expand under $\si_d$, a 
contradiction. Thus, $m=m'$. We conclude that $\Ii$, and,
accordingly, $m$, as a function of $\ell$, is unique.
\end{proof}

From now on we denote the minor from Lemma \ref{invariant:unique} by $m(\ell)$.

\begin{lem}\label{l:lineup}
The family of prime minors forms the boundary of a closed convex set
$Q=\ch(Q\cap\uc)$. Moreover, $Q\cap\uc$ is a Cantor set in which fixed
points in $\bd(\umc)$ are not isolated from either side.
\end{lem}

\begin{proof}
We claim that prime minors are pairwise disjoint. Indeed, the only way
they intersect is when two rational prime minors $m$ and $m'$ have a
common endpoint $z$ which is the terminal point of, say, $m$ and the
initial endpoint of $m'$. Then $z$ is a periodic point whose orbit $B$
contains the remaining two endpoints of $m$ and $m'$. Hence $m$ and
$m'$ are edges of the same gap $\ch(B)$. However, by the above all
edges of $\ch(B)$ are of distinct length, a contradiction.

As we showed above, prime minors are pairwise non-comparable in the
sense of $\succ$, and pairwise disjoint. Let us show that any non-fixed
point $y\in \uc$ is either a prime irrational minor, or has a prime
rational minor $m$ such that $y\in H(m)$. Indeed, choose two
consecutive fixed point $u$ and $v$ such that the positively oriented
arc $I$ from $u$ to $v$ contains $y$. Then choose an admissible
critical chord $\ell$ with $i(\ell)\in I$ such that $\si_d(\ell)=y$.
Set $m=m(\ell)$. It follows that $y\in H(m(\ell))$ as desired.

We claim that the union $T$ of prime irrational minors and the
endpoints of prime rational numbers is a Cantor set. Indeed, if a
sequence of prime minors converges to a point $y\in \uc$ then either
$y$ is an prime irrational minor, or $y$ is an endpoint of a prime
rational number because if $y$ is ``under'' a prime rational minor $m$
then the minor that converge to $y$ will also be ``under'' $m$, a
contradiction. Thus, $T$ is closed. On the other hand, an isolated
point $z\in T$ may exist only if two prime rational minors meet at $z$,
a contradiction. Since $T$ cannot contain subsegments, $T$ is a Cantor
set. Because of the symmetry, if a fixed point $a$ had been isolated
from one side, it would have been isolated from either side, a
contradiction. It remains to set $Q=\ch(T)$.
\end{proof}

\begin{defn}\label{unicritical:main:cardiod}

The set $Q$ from Lemma \ref{l:lineup} is the \emph{unicritical analog of the Quadratic Main
Cardioid}. We call it the \emph{Unicritical Main Cardioid} and denote it by $\umc$.

\end{defn}
\begin{thm}\label{t:umc}
The set $\umc$ is a gap of $\uml_d$.
\end{thm}

\begin{proof}
The edges of $\umc$ are rational prime minors. It also includes prime
irrational minors as its ``stand alone'' vertices. To show that $\umc$
is a gap of $\uml_d$ it suffices to show that there are no minors
inside $\umc$. Suppose that $m=m(\lam)$ is such a minor. By Lemma
\ref{l:minor-fxpt} we can choose the major $M_m=\ol{yz}$ of $\lam$ such
that $M_m\succ m$ and $\si_d(\ell)\in \ol{H(m)}$ for any critical chord
$\ell$ with $\ell\succ M_m$ and $|\ell|=\frac{1}{d}$. By Lemma
\ref{l:minor-fxpt}, $m=\si_d(M_m)=\ol{\si_d(z) \si_d(y)}$. Set
$\ell=\ol{y (y+\frac{1}{d})}$ and consider the minor $m(\ell)$ and the
associated major $M_{m(\ell)}=M(\ell)$.

By definition, $\ell\succ M(\ell)$. By construction, $\ell\succ M_m$.
By Theorem \ref{minors:dontcross} $m$ and $m(\ell)$ do not
cross. Hence, $M(\ell)$ and $M_m$ do not cross either. If $M(\ell)$ is ``under''
$M_m$, then $m$ is ``under'' $m(\ell)$, a contradiction with
$m$ being inside $\umc$. Thus, the only possibility left is
that $M(\ell)$ has $y$ as the initial endpoint.
Then $m$ and $m(\ell)$ share an endpoint
$\si_d(y)$. Since $m$ is inside $\umc$, $m\succ m(\ell)$
and $M(\ell)\succ M_m$. The point $\si_d(y)$ belongs to the minimal set
$\Ii(\ell)$ and has the orbit contained in $\ol{H(\ell)}$.

Since $m$ is inside $\umc$, the point $\si_d(z)$ belongs to a minor $m'\ne
m(\ell)$ (otherwise $m=m(\ell)$). Hence for some $N\ge
2$ the point $\si_d^N(z)$ does not belong to $\ol{H(\ell)}$. We
conclude that $\si_d^N(M_m)$ is a chord with one endpoint
$\si_d^N(y)\in \ol{H(\ell)}$ and the other endpoint $\si_d^N(z)\notin
\ol{H(\ell)}$. However this contradicts Corollary
\ref{Corr:CentralStrip} which describes the dynamics of majors of
laminations. This completes the proof.
\end{proof}

One can transfer Theorem \ref{t:umc} onto Fatou gaps.

\begin{lem}\label{countable:gaps}
Let $\Cc(\lam)$ is a critical Fatou gap of a unicritical lamination
$\lam$ of period $n$. Let $U=\si_d(\Cc(\lam))$. Let $\psi:\bd(U)\to
\uc$ be a monotone map that collapses edges of $U$ to points. Then
$\uml_d$ contains a gap $\umc_U$ whose boundary is
$\psi^{-1}(\bd(\umc))$. In particular, $\umc_U\cap \uc$ is a Cantor
set, and the minor $m$ of $\lam$ has countably infinitely many
immediate $\succ$-successors.
\end{lem}

\begin{proof}
The existence of the gap $\umc_U$ is immediate from Theorem \ref{t:umc}
and properties of the map $\psi$ which semiconjugates $\si_d^n|_U$ and
$\si_d$. The only thing that needs to be noticed is that when we lift
$\bd(\umc)$ to $\umc_U$ using $\psi$, one fixed point in $\bd(\umc)$
corresponding to the minor $m$ of $\lam$ is lifted to a leaf $m$ which
creates isolation on the appropriate sides of the endpoints of $m$ in
$\umc_U\cap \uc$. However by Lemma \ref{l:lineup} fixed points in
$\umc$ are non-isolated on either side, hence the endpoints of $m$ are
non-isolated in $\umc_U\cap \uc$ as desired.
\end{proof}

Abusing the language, we will call a gap similar to $\umc_U$ from Lemma
\ref{countable:gaps} a \emph{copy} of $\umc$.

\begin{lem}\label{l:molecule}
If $m$ is a periodic minor then there are two copies of $\umc$ that
share $m$. Any non-degenerate minor with a periodic endpoint is a
periodic minor.
\end{lem}

\begin{proof}
Our analysis implies that if $m$ is a non-degenerate prime minor of
period $n$, then in the associated lamination $\lam$ there exists a
finite invariant gap $G$ such that $m$ is one of its edges. The
corresponding major $M=M_m\succ m$ is the edge of $G$ that cuts the
rest of $G$ from the center of the circle (equivalently, from all fixed
points of $\si_d$). At $M$ there is a critical Fatou gap $\Cc(\lam)$ of
$\lam$ of period $n$. Set $U=\si_d(\Cc(\lam))$. Then, by Lemma
\ref{countable:gaps} there is a counterpart $\umc_U$ of $\umc$ in $U$
such that $m$ corresponds to the degenerate minor in $\umc$ associated
with angle $0$, and otherwise the correspondence between $\umc$ and
$\umc_U$ is a homeomorphism. Moreover, $\umc_U$ is itself a gap of
$\uml_d$.

It follows that if we repeatedly apply the same construction, we will
discover countable concatenations of copies of $\umc$ whose existence
is established in Lemma \ref{countable:gaps}. This completely describes
all minors with periodic endpoints. Indeed, let $m=\ol{ab}=m(\lam)$ be
the minor of a unicritical lamination $\lam$ and $a$ be a periodic
point of period $k$. By Lemma \ref{nontrivial:object} there exists a
critical Fatou gap $W$ of period $k$ such that $\si_d(W)$ has an edge
$\ell$ of period $k$ and an endpoint $a$. By the previous paragraph,
$\ell$ is a minor and there are two copies of $\umc$ adjacent at each
other along $\ell$. Moreover, the boundaries of these copies of $\umc$
intersected with $\uc$ are Cantor sets. It follows that $m=\ell$ as in
$\uml_d$ there is no room for any other minor with endpoint $a$.
\end{proof}


\begin{defn}\label{d:towards}
Let $\ell\ne \ell'$ be leaves of a lamination $\lam$. If $\si_d(\ell)$ and $\ell'$ are contained in the same component of $\cdisk\sm \ell$ we say that $\ell$ maps \emph{toward} $\ell'$, otherwise that it maps \emph{away from} $\ell'$.
\end{defn}

We will need the next lemma which is in fact a version of the fixed
point theorem from continuum theory.

\begin{lem}\label{l:fxpt-1}
Suppose that $\ell'$ and $\ell$ are two disjoint leaves of a
lamination $\lam$. Moreover, suppose that $\si_d^k$ maps both leaves either toward each other or away from each other. Then there exists a $\si^k_d$-invariant gap or leaf
$G$ that separates $\ell'$ and $\ell$. If $\ell$ and $\ell'$ map away from each other then there exists a leaf $\oy$ that separates them and maps to itself.
\end{lem}

\begin{proof}
Call leaves of $\lam$ that separate $\ell'$ and $\ell$ \emph{separating}. Separating leaves are ordered, say, from $\ell'$ to $\ell$ family. By our assumptions there must exist a place in this family where the direction of the movement of leaves switches. Thus, either there exists a separating leaf that maps to itself, or there exists a gap $G$ separating $\ell'$ and $\ell$ such that $\si_d^k(G)=G$. If leaves $\ell'$ and $\ell$ map away from each other, then the latter option is impossible unless there is an edge of $G$ that maps to itself because otherwise $G$ would have strictly covered itself under $\si_d^k$. 
\end{proof}

Let us now study properties of a minor $m$ as a leaf of $\uml_d$.

\begin{lem}\label{l:root}
Let $m\in \uml_d$ be a non-degenerate non-periodic minor. If there is a
sequence of leaves $\ell_i\in \lam(m)$ with $\ell_i\succ m$ and
$\ell_i\to m$, then $m$ is the limit of periodic minors $n_j$ with
$n_j\succ m$ for any $j$.
\end{lem}

\begin{proof}
By Thurston's pullback construction the assumptions imply that we may
assume that all $\ell_i$'s are pullbacks of majors. Hence there is a
sequence of numbers $k_i\to \infty$ such that $\si_d^{k_j}(\ell_j)=m$.
On the other hand $\si_d^{k_j}(m)\succ m, \si_d^{k_j}(m)\ne m$ by
properties of minors and because $m$ is not periodic. Thus,
$\si_d^{k_j}$ maps $\ell_j$ and $m$ in opposite directions.

By Lemma \ref{l:fxpt-1} there exists a $\si_d^{k_j}$-invariant gap or
leaf of $\lam(m)$ that separates $m$ and $\ell_j$. Now, consider the
gap (or leaf) $G$. It has a leaf $\bar g$ that ``faces'' $m$ (so that
$\bar g$ separates $m$ and the rest of $G$). We may also assume that
the two arcs between the endpoints of $\bar g$ and the endpoints of $m$
are shorter than $|m|$. Consider the forward orbit of $\bar g$. By
Lemma \ref{oscillating} the length of any image of $\bar g$ is at least
$|m|$. Hence $\bar g\succ \si_d^i(\bar g)$ for some $i$ then
$\si_d^i(\bar g)$ separates $g$ and $m$ (any other location of
$\si_d^i(\bar g)$ would imply that the length of $\si_d^i(\bar g)$ is
smaller than $|m|$). Since $\bar g$ is periodic, we may assume that
some image $\bar h$ of $\bar g$ separates $m$ and $G$ and has no images
$\si_d^j(\bar h)$ such that $\bar h\succ \si_d^i(\bar h)$.

It is easy to see that the $d$ pullbacks $H_1, \dots, H_d$ of $\bar h$
that are leaves of $\lam(m)$ are located outside of the Central Strip
$C(m)$. Forward images of $\bar h$ cannot cross those pullbacks. By
Lemma \ref{condition:minor} the leaf $\bar h$ is a minor. Since it is
periodic, it is a part of a copy of $\umc_U$ of $\umc$ for the
appropriate Fatou gap $U$.
\end{proof}

Consider now 
non-degenerate and non-periodic minors.

\begin{lem}\label{l:ndnp-minor}
Let $\lam$ be a unicritical lamination with non-degenerate and
non-periodic minor. Then there exists a lamination $\lam'\subset \lam$
with finite critical set $\Cc(\lam)$ such that the set
$\si_d(\Cc(\lam'))=H'$ is approached by leaves of $\lam'$ from all
sides and one of the following two cases holds:

\begin{enumerate}

\item $\lam'=\lam$;

\item the set $\Cc(\lam')$ is a gap with $dn, n\ge 3$ edges and
    $\lam$ is obtained from $\lam'$ by inserting $d$ sibling edges
    in $\Cc(\lam')$ so that $\lam$ has a critical $2d$-gon $G$
    inside $\Cc(\lam')$ and then pulling $G$ back along the
    backward orbit of $\Cc(\lam')$ in $\lam'$.
\end{enumerate}

\end{lem}

\begin{proof}
We claim that the critical gap $\Cc(\lam)$ of $\lam$ is finite. Indeed,
if $\Cc(\lam)$ is infinite then it is periodic of period, say, $N$;
thus, $\Cc(\lam)$ is an infinite gap that maps onto itself $d$-to-$1$
under $\si_d^N$. Yet then by Lemma \ref{l:prim-refixed} the associated
minor must be periodic, a contradiction. Hence $\Cc(\lam)$ is a finite
gap with $dn$ edges, and $n\ge 2$ ($n=1$ would imply that the minor $m$
is degenerate, a contradiction).

Consider several cases. Suppose than $n\ge 3$ so that
$\si_d(\Cc(\lam))=H$ is a gap (not a leaf). We claim that then each
edge of $H$ is a limit edge of leaves of $\lam$. Indeed, if $\ell$ is
an edge of $H$ isolated in $\lam$, then there is a gap $G$ that shares
$\ell$ with $H$. Since $\Cc(\lam)$ is a finite gap, then $G$ cannot be
infinite because any infinite gap must have an image on which $\si_d$
is not 1-to-1 while the only gap of $\lam$ on which $\si_d$ is not
1-to-1 is a finite gap $\Cc(\lam)$. Thus, $G$ is finite.

Suppose that $G$ at some point maps onto $\Cc(\lam)$. Then on the next
step $G$ maps to $H$; we may assume that $\si_d^N(G)=H$ for some $N$.
it follows that $G$ is a $dn$-gon while $H$ is an $n$-gon. Therefore
$\si_d^N$ cannot act on $G$ and $H$ ``swapping'' them (in other words,
$\si_d^N$ cannot be ``flipping'' $\ell$). Properties of laminations
imply that $\si_d^N$ cannot map $\ell$ to itself as identity either.
So, $\si_d^N$ maps $\ell$ to an edge $\ell'\ne \ell$ of $H$. Properties
of laminations imply now that $H$ maps under $\si_d^N$ to a gap
adjacent to $H$ along $\ell'$, etc. Thus, iterating $\si_d^N$ and using
properties of invariant laminations we will find an infinite
concatenation of images of $H$ which is clearly impossible (e.g.,
because $H$ is (pre)periodic by Lemma \ref{triangles:wander}).

So, $G$ is not precritical; hence, by Lemma \ref{triangles:wander} $G$
is preperiodic and never maps to $H$. The same can be said about $H$,
but for an even more trivial reason: $H$ has $n$ edges and $\Cc(\lam)$
had $dn>n$ edges, hence $H$ is not precritical and by Lemma
\ref{triangles:wander} $H$ is (pre)periodic. As we apply $\si_d$ over
and over to the union $G\cup H$, both gaps $G$ and $H$ stay away from
the critical gap $\Cc(\lam)$ and eventually map onto two distinct
adjacent finite periodic gaps $\widehat G$ and $\widehat H$ that share
a leaf $\hat \ell$. Let us show that this is impossible. Indeed, by
Lemma \ref{transitivity} at every edge of $G'$ there is an image of
$H'$ attached to $G'$. Then we can say the same about each such image
of $H'$, and so on. This yields infinite family of images of $\widehat
G$ or $\widehat H$, a contradiction. So, if $H$ is a gap, then all
edges of $H$ are limit edges in $\lam$. Then we can set $\lam'=\lam$
and case (1) holds.

Otherwise $\Cc(\lam)$ is a $2d$-gon and $\si_d(\Cc(\lam))$ is a leaf
$m$. If $m$ is a limit leaf from both sides, then, again, we can set
$\lam'=\lam$ and case (1) holds. Suppose that $H=m$ is an edge of a gap
$G$ of $\lam$. As before, $G$ cannot be infinite because an infinite
gap must have an image which maps forward no 1-to-1, and the only way
it is possible under the circumstances is when an image of $G$ equals
$\Cc(\lam)$, a contradiction since $\Cc(\lam)$ is a $2d$-gon. This, $G$
is a finite gap.

We claim that $G$ never maps onto $\Cc(\lam)$. Indeed, if it happens
then $m$ must be periodic, a contradiction; the same reason implies
that $G$ never maps onto itself. By Lemma \ref{triangles:wander} $G$ is
preperiodic. There are $d$ immediate preimages of $G$ denoted $G_1,
\dots, G_d$; these are gaps attached to the appropriate edges of
$\Cc(\lam)$ denoted $\bar g_1, \dots, \bar g_d$, resp. Then $\bar g_1,
\dots, \bar g_d$ are isolated sibling leaves of $\lam$ with the same
image $m$. Evidently, if we remove their backward obits from $\lam$ we
will obtain a new \emph{closed} (because $\bar g_i$'s are isolated)
invariant lamination $\lam'\subsetneqq \lam$ in which $G$ remains a gap
(no edge of $G$ belongs to the union of the backward orbits of $\bar
g_i$'s because $G$ is not periodic). By construction, the critical gap
$\Cc(\lam')$ of the new lamination $\lam'$ coincides with the convex
hull of $\bigcup G_i$ or, equivalently, is the union of $\Cc(\lam)$ and
$\bigcup G_i$ from which all leaves $\bar g_i$ are removed.

Now, the just described critical set $\Cc(\lam')$ of the new lamination
$\lam'$ is a finite gap with $dk$ edges and $k\ge 3$. Hence by the
above proven, all edges of $G=\si_d(\Cc(\lam'))$ are limit edges of
$\lam'$ as claimed.
\end{proof}

Finally, we can prove that $\uml_d$ is a q-lamination.

\begin{thm}\label{t:umlq}
The lamination $\uml_d$ is a q-lamination.
\end{thm}

\begin{proof}
Consider a minor $m=m(\lam)\in \uml_d$. If it is periodic then by Lemma
\ref{l:molecule} it is isolated in $\uml_d$ and is disjoint from all
other minors. Assume that endpoints of $m$ are non-periodic. By Lemma
\ref{l:ndnp-minor} there are two possibilities. First, $m$ is the limit
leaf of leaves of $\lam(m)$ from both sides. Then by Lemma \ref{l:root}
and Lemma \ref{l:partialorder}(1) $m$ is the limit of minors from both
sides. Finally, assume that $m=m(\lam)$ is not the limit of leaves of
$\lam$ from one side. Then by Lemma \ref{l:ndnp-minor} there exists a
lamination $\lam'$ with a critical gap $\Cc(\lam')$ with $dn, n\ge 3$
edges such that $m$ is an edge of the gap $\si_d(\Cc(\lam'))$ which is
approached by leaves of $\lam'$ from the outside of
$\si_d(\Cc(\lam'))$. By Lemma \ref{l:partialorder}(1) $m$ is a minor,
approached by other minors from the outside of $\si_d(\Cc(\lam'))$.
Thus, any non-periodic minor is either itself approximated from all
sides by other minors, or is an edge of a finite gap whose all edges
are approached ny minors from the outside of  $\si_d(\Cc(\lam'))$.
Evidently, all this implies that $\uml_d$ is a \emph{q-lamination}.
\end{proof}

Lemma \ref{l:fc} is a $\uml_d$-version of the density of hyperbolicity.


\begin{lem}[Fatou conjecture]\label{l:fc}
Copies of $\umc$ are dense in $\uml_d$.
\end{lem}

\begin{proof}
It suffices to prove that periodic minors are dense in $\uml_d$. Let
$m$ be a non-periodic  non-degenerate minor. By Lemma \ref{l:root} if
$m$ is the limit of leaves $\ell_i\succ m $ then $m$ is the limit
of periodic minors $\ell'_i$ such that $\ell'_i\succ m$. Now, if $m$ is
not the limit of leaves $\ell_i\succ m $ with $\ell_i\in \lam(m)$
then $m$ is an edge of a finite gap $G$ in $\uml_d$ and no edge of this
gap is ``under'' $m$. Then $m$ is the limit of minors all of which are
``under'' $m$. Moreover, these minors can be chosen to be limits of
leaves of their respective laminations located so that the minors are
``under'' those leaves; hence, again by Lemma \ref{l:root}, we can find
periodic minors in a small neighborhood of $m$ located ``under'' $m$.
\end{proof}

\section{Constructing $UML_{d}$}\label{Facts}

Let $\mathcal{L}$ be a unicritical $\si_d$-invariant lamination with a
unique critical set $\Cc(\lam)=\mathcal{C}$. If the $\si_d$-orbit of
$\mathcal{C}$ is finite, then $\Cc$ is either periodic or strictly
pre-periodic. In the latter case we call the lamination $\mathcal{L}$ a
\emph{Misiurewicz lamination} and $\mathcal{C}$ a \emph{Misiurewicz
critical set}.

\begin{lem}\label{l:mis-lam}
Let $\mathcal{L}$ be a Misiurewicz lamination with critical set
$\mathcal{C}(\lam)$. Then there exists a q-lamination $\lam'\subset \lam$ with critical set $\Cc(\lam')$ such that all gaps of $\lam'$ are finite, $\Cc(\lam)\subset \Cc(\lam')$, and one of the following holds:

\begin{enumerate}

\item $\Cc(\lam)$ is an all-critical $d$-gon and $\Cc(\lam)\subset \Cc(\lam')$;

\item $\Cc(\lam)$ is a critical $2d$-gon that map onto an edge of $\si_d(\Cc(\lam'))$ and $\Cc(\lam)\subset \Cc(\lam')$;

\item $\Cc(\lam)=\Cc(\lam')$.

\end{enumerate}

In particular, all leaves of $\lam'$ are non-isolated from at least one side.
\end{lem}

\begin{proof}
Clearly, $\mathcal{C}$ must be finite (if $\mathcal{C}$ is infinite
then $\mathcal{C}$ is periodic but we have assumed that $\mathcal{C}$
is strictly pre-periodic). Also, all gaps of $\mathcal{L}$ are finite.
Indeed, if $G$ is an infinite gap of $\mathcal{L}$, then there exists a
critical forward iterate $U$ of $G$ contradicting the fact that
$\mathcal{C}$ is finite.

If $G$ and $H$ are gaps that never map to $\mathcal{C}$, then $G$ and
$H$ cannot share a common leaf. Indeed, if $G$ and $H$ share a common
leaf $\ell$, we can $\si_d$ over and over to the union $G \cup H$ and
see that $G$ and $H$ stay away from $\mathcal{C}$ and eventually map
onto two distinct adjacent finite periodic gaps $G'$ and $H'$ that
share a leaf $\ell'$. By Lemma \ref{transitivity} at every edge of $G'$
there is an image of $H'$ attached to $G'$, at every edge of these
images of $H'$ there is an image of $G'$, and so on. This yields
infinite family of concatenated images of $G'$ and $H'$, a
contradiction.

Finally, if $G$ is an image of $\mathcal{C}$, then by
Lemma \ref{l:ndnp-minor} only the three cases from Lemma \ref{l:mis-lam} are possible.
\end{proof}

Define an equivalence relation $\approx_\lam$ on $\uc$ by declaring
that $x\sim y, x, y\in \uc$ if a finite concatenation of leaves of
$\mathcal{L}$ joins $x$ and $y$ (see Definition
\ref{relation:lamination}). It is well-known (and easy to see) that
$\approx_\lam$ is a laminational equivalence relation (see the remark
after Definition \ref{relation:lamination}). Let $\psi:\uc\to
\uc/\approx_\lam=J_{\approx_\lam}=J$ be the quotient map that
semiconjugates $\si_3:\uc\to \uc$ and the induced topological
polynomial $f_\lam=f:J\to J$.

If $\mathcal{L}$ is a Misiurewicz lamination with critical set
$\mathcal{C}$ then we set $\Cc'=\psi(\Cc)$ and observe that the
$f$-orbit $O(\mathcal{C}')$ of $\mathcal{C}'$ is finite and contained
in $J$. Moreover, by Lemma \ref{l:mis-lam} all gaps of $\lam$ are
finite and, hence, $J$ is a \emph{dendrite} (i.e. it is locally
connected and contain no subsets homeomorphic to the circle since such
a subset would correspond to an infinite gap of $\lam$). This allows us
to use  a standard interval notation in $J$ and denote by $[a, b]$ the
unique interval with endpoints $a, b\in J$. Let
$\mathcal{H}(\mathcal{L})$ be the smallest connected set containing
$O(\mathcal{C}')$ in $J$. Then $\mathcal{H}(\mathcal{L})$ is a tree
(because $O(\mathcal{C}')$ is finite) $\mathcal{H}(\mathcal{L})$ called
the \emph{Hubbard tree (of $\mathcal{L}$)}.

By Lemmas \ref{nontrivial:object} and \ref{l:mis-lam}, there exists an
invariant finite gap or non-degenerate leaf $G$ of $\lam$. It
corresponds to a fixed point $\psi(G)\in J$. The map $\si_d|_G$ is a
transitive rotation of period $n>1$ (by Lemma \ref{transitivity} this
can fail only if $G$ is a leaf with fixed endpoints, but each hole of
$\Cc$ is of length at most $\frac{1}{d}$ and cannot contain two
$\si_d$-fixed points, a contradiction). This implies that if $I$ is a
short subinterval of the interval $[\Cc', \psi(G)]$ with endpoint
$\psi(G)$ then $I, f(I), \dots, f^{n-1}(I)$ are pairwise disjoint
(except for the common endpoint $\psi(G)$ of them all) and
$f^n(I)\supsetneqq I$ (the latter follows from the fact that locally
$f$ expands because $\si_d$ is locally expanding).  Hence $\psi(G)\in
\Hc(\lam)$.

Clearly, $\Hc(\lam)$ is the union all intervals $T=[f^i(\Cc'),
f^j(\Cc')]$, and there are finitely many of them. If $\Cc'\notin T$
then $f(T)=[f^{i+1}(\Cc'), f^{j+1}(\Cc')]\subset \Hc(\lam)$. If
$\Cc'\in T$ then $f(T)=[f^{i+1}(\Cc'), f(\Cc')]\cup [f(\Cc'),
f^{j+1}(\Cc')]\subset \Hc(\lam)]$. Hence $f(T)\subset T$. To show that
$f(T)=T$, note that $\Hc(\lam)$ is the union of all intervals
$[f^i(\Cc'), \psi(G)]$ (this is because by the previous paragraph
$\psi(G)\in \Hc(\lam)]$). Then each $[f^i(\Cc'), \psi(G)]$ is covered
by the image of $f^{i-1}(\Cc'), \psi(G)]$, and it remains to show that
some interval $[f^k(\Cc'), \psi(G)]$ contains $\Cc'$. Indeed, otherwise
$O(\Cc)$ is contained in the hole of $\Cc$ that contains $G$. However
$G$ is the unique minimal subset of that hole. Hence the periodic
portion of $O(\Cc)$ must in fact coincide with $G$. However all
preimages of $G$ are separated from $G$ by the set $\Cc$ which implies
the desired. We have proven the following lemma.

\begin{lem}\label{l:hub-tree}
We have that $f(\Hc(\lam))=\Hc(\lam)$ and $\psi(G)\in \Hc(\lam)$.
\end{lem}

The quadratic version of the next lemma is due to Lavaurs \cite{lav89}.

\begin{thm}\label{Lavaurs:lemma}
If two leaves in $\uml_d$ have the same period and are $\succ$-comparable,
then these two leaves are separated in $\udisk$ by a leaf in $\uml_d$
of lower period.
\end{thm}

\begin{proof}
Let $m\succ n$ be minors of period $p$; then their endpoints are of period $p$. Choose a preperiodic point
$z\in H(n)$ such that the period of a periodic eventual image of $z$ is
larger than $p$. By Theorem \ref{t:umlq} and Lemma \ref{l:mis-lam}
there exists a Misiurewicz q-lamination $\lam$ and a (possibly
degenerate) preperiodic leaf or gap $\delta$ of $\uml_d$ such that
$z\in \si_d(\Cc(\lam))=\delta$ and $\delta\cap \uc\subset H(n)$. By
Lemma \ref{l:partialorder}(2) $m$ and $n$ are leaves of $\mathcal{L}$.
Set $\Cc(\lam)=\Cc$, $\psi(\mathcal{C}) = \mathcal{C'} $, $\psi(m)=m'$
$\psi(n)=n'$ and $\psi(\delta)=\delta'$; these are points of $\mathcal{H(\mathcal{L})}$. Then
$\mathcal{H(\mathcal{L})}\sm \mathcal{C}'$ has $ld$ components for some
$l\ge 1$.

We claim that $f^k(m')$ and $f^k(n')$ lie in different components of
$\mathcal{H(\mathcal{L})}\sm \mathcal{C}'$ for some $k>0$. Indeed if,
for all $k>0$, $f^k(m')$ and $f^k(n')$ lie in the same component of
$\mathcal{H(\mathcal{L})}\sm \mathcal{C'}$, then $m$ and $n$ are edges
of the same periodic gap of the lamination $\mathcal{L}$. Then, by
Lemma \ref{transitivity}, there exists $i$ such that $\si_d^i(m)=n$
contradicting the fact that a minor cannot map under itself. Hence,
$f^k(m')$ and $f^k(n')$ lie in different components of
$\mathcal{H(\mathcal{L})}\sm \mathcal{C'}$ for the minimal $k>0$. Then,
by continuity of $f$, there exists a gap $\Cc_{-1}$ that separates $m$
and $n$ and is such that $\si_d^k(\Cc_{-k})=\Cc$. Moreover, $k$ is the
least number for which such gap exists. Since $m$ and $n$ are
$p$-periodic, then $k\le p-1$. Consider two cases.

		

(1) Let $k<p-1$. Then $f^{k+1}(\mathcal{C}_{-k}')=\delta'$
and so $\Cc_{-k}$ maps away from $m$; also,
by the properties of minors $m$ maps away from $\Cc_{-k}$. Hence by Lemma \ref{l:fxpt-1} there is a leaf $t$ of $\lam$ that separates $m$ and $\Cc_{-k}$ and maps to itself under $\si_d^{k+1}$. Let us now choose the lest number $j$ such that there exists a leaf $t'$ of period $j$ that separates $m$ and $n$. It follows that $j\le k+1<p$. Moreover, we may assume that $t'$ is chosen among all such leaves as the closest to $m$.
We claim that $t'$ is a minor. Indeed, the only reason why $t'$ may not be a minor is if $\si_d^i(t')$ is ``under'' $t'$ for some $i<j$. However then by Lemma \ref{l:fxpt-1} there exists a leaf $t''$ that separates $m$ and $t'$ and is such that $\si_d^i(t'')=t''$, a contradiction with the choice of $t'$. So, $t'$ is a minor as desired.	


		


(2) Let $k=p-1$. Then $f^{k+1}$ maps $n'$
to itself while mapping $\mathcal{C}_{-k}'$ ``under''
itself. Since $f^{k+1}$ is monotone on $[n', \mathcal{C}_{-k}']$, a small subinterval $I=[y, m']$ of $[\mathcal{C}_{-k}', m']$ ``flips'' under the action of $f^p$ which means that the leaf $m$ flips under the action $\si_d^p$, a contradiction with the fact that the endpoints of $n$ are $f^p$-fixed. Thus, case (2) is impossible.
\end{proof}

Recall that by Lemma \ref{firstreturn:dfixed} a critical $n$-periodic Fatou gap $U$ of a unicritical lamination $\lam$ 
has one $n$-periodic major and otherwise $d-2$ periodic points on its boundary that are not endpoints of any edge of $U$. 
Therefore, the minor $m(\lam)$ of $\lam$ is a unique $n$-periodic non-degenerate edge of $\si_d(U)$ while all other $n$-periodic 
points from $\bd(\si_d(U))$ are disjoint from all edges of $\si_d(U)$. 
In the forthcoming arguments we will use the semiconjugacy $\psi:\bd(\si_d(U))\to \uc$ that collapses all edges of $\si_d(U)$ 
to points and seminconjugates $\si_d^n|_{\si_d(U)}$ and $\si_d$.

\begin{lem}\label{l:fatou}
Let $\lam$ be a unicritical lamination with critical Fatou gap $U$ of period $n$. Then all non-degenerate edges of $\si_d(U)$ are non-degenerate minors.
In particular, periodic minors are either disjoint from $\si_d(U)$< or contained in it and then their periods are multiples of $n$. Moreover, all points of $\bd(\si_d(U))$ that are of period $n$ and are not the endpoints of the minor $m(\lam)$ are degenerate minors.
\end{lem}

\begin{proof}
Let $\ell$ be a non-degenerate edge of $\si_d(U)=V$. Consider the convex hull $C$ of $\si_d^{-1}(\ell)$. Clearly, $C$ and its forward orbit (which begins with $\ell$) is a forward invariant lamination. Then Thurston's pullback construction yields a lamination with minor $\ell$ as desired. It follows that if a minor $m'$ is non-disjoint from the interior of $V$ then in fact $m'\subset V$. One may say that the gap $V$ with all the minors contained in it forms a part of $\uml_d$. Since all edges of $V$ are pullbacks of $m(\lam)$, the claim about periodic minors from the lemma easily follows.

Let $x\in \bd(V)$ be a point of period $n$ but not an endpoint of $m(\lam)=m$ ($m$ is a non-degenerate edge of $V$ such that $\si_d^n(m)=m$). Assume, by way of contradiction, that there exists a non-degenerate minor $\oy=m(\lam')$ with an endpoint $x$ (here $\lam'$ is some unicritical lamination). By the previous paragraph $\oy\subset V$. Construct the lamination $\lam(\oy)$ as usual: first construct the critical set which is the convex hull $\Cc(\oy)$ of $\si_d^{-1}(\oy)$ and then pull it back in the compatible with itself way.

Evidently, all pullbacks of $\Cc(\oy)$ as well as images of $\oy$ are contained in the corresponding images and pullbacks of $V$ that are gaps of $\lam$. The same applies to the leaves from the closure of these pullbacks of $\Cc(\oy)$ and images of $\oy$, i.e. to all leaves of $\lam(\oy)$. Thus, the restriction $\lam(\oy)|_V$ of $\lam(\oy)$ onto $V$ satisfies all properties of invariant laminations except that instead of $\si_d$ we need to consider $\si_d^n$. Applying $\psi$ to it, we will obtain a $\si_d$-invariant unicritical lamination with $\psi(x)$ belonging to the image of its critical set. However then by Lemma \ref{l:fxpt-lami} no leaf of this lamination may have an endpoint $\psi(x)$, a contradiction with the existence of the leaf $\psi(\oy)$.
\end{proof}


Now, suppose that for some $k$ all periodic minors of periods at most $k$, non-degenerate and degenerate, are already constructed. Denote their union by $A_k$. Take a component $B$ of $\cdisk\sm A_k$.
Let $z$ be a periodic point of period $k+1$ that belongs to $B$.
By Lemma \ref{nontrivial:object} there is a lamination $\lam$ with a critical Fatou gap $U$ of period $k+1$ such that $z\in \si_d(U)=V$. By Lemma \ref{l:fatou} the entire $V$ is contained in $B$. 
Either $z$ is an endpoint of the minor $m(\lam)$, or $z$ is a degenerate minor located ``under'' $m(\lam)$. All of this can happen more than once inside $B$ and will exhaust all periodic points of period $k+1$ that belong to $B$.


We need to make another observation based upon our study of the Unicritical Main Cardioid $\umc$. By construction, leaves from $\umc$ of periods at most $k$ are included in $A_k$, and no minor intersects the interior of $\umc$. Hence each minor of period $k+1$ contained in $B$ is located between two consecutive $\si_d$-fixed points, and the same applies to any gap similar to the gap $V$ considered above. 
Together with the previous paragraph, this yields an algorithm for constructing $\uml_d$. The quadratic version of this algorithm was proved by Lavaurs in \cite{lav89}.




\section{The Lavaurs Algorithm in the Unicritical Case}
\label{Sec:Lavaurs}

Let us construct leaves of$\uml_d$ that have periodic endpoints; the closure of their collection is $\uml_d$. The map $\sigma_d$ has $d-1$ fixed points on the circle: $0,\frac{1}{d-1},\frac{2}{d-1},\dots,\frac{d-2}{d-1}$. By
Lemma \ref{l:fxpt-lami} they are \emph{degenerate minors}. We place a \emph{dot} on the circle at each of these points.

Now, take points of \emph{period} 2 under $\si_d$. The map $\sigma_d$ has $d^2-1$ points of period $2$ on the circle
given by $\{\frac{k}{d^2-1}|k=0,1,2\dots,d^2-1\}$. Some of these will correspond to \emph{degenerate minors} and some of them will be the endpoints of \emph{non-degenerate minors}. It follows that for any \emph{non-degenerate} minor $\gamma$ of period 2, there exist $d-2$ \emph{degenerate} minors $\delta_i$ such that $\gamma \succ \delta_i$ for every $i$. Moreover, any non-degenerate minor of period $2$ is located between two consecutive fixed points. Take the smallest point of period $2$: $\frac{1}{d^2-1}$, connect it to the point $\frac{d}{d^2-1}$, and place a \emph{dot} at the intermediate points $\frac{2}{d^2-1},\frac{3}{d^2-1},\dots, \frac{d-1}{d^2-1}$ (by Lemma \ref{l:fatou}
they are \emph{degenerate minors}). Continue this way around the circle.

In general, suppose the periodic points  on the circle of period
less than $k$ have been \emph{connected} or \emph{dotted}. Connect those of period $k$ starting with the smallest point $p$ of period $k$.
Let $q_1,q_2,\dots,q_{d-1}$ be the next $d-1$
points \emph{ordered anticlockwise} and belonging to the same component of $\cdisk$ with minors of periods less than $k$ removed as $p$. Connect $p$ to $q_{d-1}$ and place
a \emph{dot} at the intermediate points: $q_1,\dots,q_{d-2}$.
Choose the next
available (not yet \emph{connected} or
\emph{dotted}) point of period $k$ and proceed in the same way;
continue similarly around the circle. We do this for periodic points on the circle in the order of increasing periods.

\begin{figure}[H]
\centering
\includegraphics[width=0.45\textwidth]{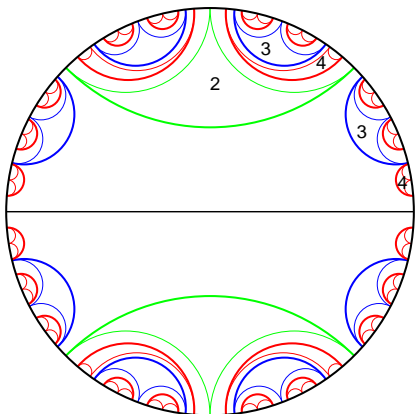}
\caption{{\textit{Construction of the periodic minors of the Unicritical
Minor Lamination $\uml_3$ up to period $4$; the gaps for different periods are coded by colors, and some periods are labeled.}}}
\label{Lavaurs:sigma3}
\end{figure}



\end{document}